\documentclass[a4paper,11pt]{amsart}
\usepackage[margin=2.5cm]{geometry}

\usepackage[T1]{fontenc}
\usepackage{lmodern, amsfonts,amsmath,amstext,amsbsy,amssymb,
amsopn,amsthm,upref,eucal,mathptmx,mathtools,url,thmtools,cleveref}

\usepackage{mathrsfs}

\RequirePackage{xcolor} 
\definecolor{halfgray}
{gray}{0.55}
\definecolor{webgreen}
{rgb}{0,0.4,0}
\definecolor{webbrown}
{rgb}{.8,0.1,0.1}
\definecolor{red}
{rgb}{1,0,0}
\usepackage{microtype}
\usepackage{tikz}

\newcommand \R {{ \mathbb R}}
\def\C{{\mathbb C}}

\newcommand \N {{ \mathbb N}}

\newcommand*{\diff}{\mathop{}\!\mathrm{d}}
\newcommand{\one}{{\rm 1\mskip-4mu l}}

\newcommand{\horo}{\mathsf{h}}
\newcommand{\geo}{\mathsf{g}}
\newcommand{\Ual}{U_{\alpha}}

\newcommand{\tc}{\mathsf{h}^{\alpha}}

\newcommand{\cP}{\mathcal{P}}
\newcommand{\cA}{\mathcal{A}}
\newcommand{\cQ}{\mathcal{Q}}
\newcommand{\cD}{\mathcal{D}}
\newcommand{\cC}{\mathcal{C}}

\newcommand{\cW}{\mathcal{W}}
\newcommand{\cF}{\mathcal{F}}

\newcommand{\volal}{\, \vol_{\alpha}}
\newcommand{\Cal}{C_{\alpha}}

\newcommand{\nal}{\|\alpha-1\|_{W^7}}
\newcommand{\xnal}{\|X\alpha\|_{W^6}}

\newcommand{\SL}{%
\operatorname{SL}
}

\DeclareMathOperator{\vol}{vol}

\DeclareMathOperator{\Spec}{Spec}
\DeclareMathOperator{\dist}{d}
\DeclareMathOperator{\Id}{Id}

\newtheorem{theorem}{Theorem}
\newtheorem {lemma}[theorem]{Lemma}
\newtheorem {proposition}[theorem]{Proposition}
\newtheorem{corollary}[theorem]{Corollary}
\newtheorem{remark}[theorem]{Remark}

\date{\today}
\author{Davide Ravotti}
\address{Université de Lille, CNRS, UMR 8524 - Laboratoire Paul Painlevé, F-59000 Lille, France}
\email{davide.ravotti@gmail.com\\}

 \title[Time-changes of horocycle flows]
 {Mixing asymptotics for time-changes of horocycle flows}

\begin{document}

\maketitle

\begin{abstract}
	Mixing-via-shearing is a powerful and versatile method for establishing mixing properties of smooth parabolic flows. In its quantitative form, it provides upper bounds on the decay of correlations for sufficiently smooth observables. Despite its wide applicability, determining the exact rates of mixing for a given smooth parabolic flow remains notoriously difficult. Apart from the classical horocycle flow, no examples are known where polynomial asymptotics, or sharp lower bounds, hold.
    In this paper, we address this question for smooth time-changes of horocycle flows on compact hyperbolic surfaces. Our approach relies on a refined version of the mixing-via-shearing method which leverages on a precise description of the ergodic integrals for horocycle flows, in particular of the regularity of the coefficients appearing in their asymptotic expansions. Using this method, we prove polynomial upper bounds on the decay of correlations for smooth observables that match the optimal rates originally obtained by Ratner for the standard horocycle flow. Furthermore, in the presence of a spectral gap below $1/4$, we establish exact polynomial asymptotics, mirroring the classical behavior of the horocycle flow.
\end{abstract}
    
\section{Introduction}

Understanding the long-term statistical behavior of smooth parabolic flows is a central problem in parabolic dynamics. 
A prominent example of parabolic flow, which encapsulates the main features of what is considered as typical \lq\lq uniformly parabolic behaviour\rq\rq, is the horocycle flow on (the unit tangent bundle of) hyperbolic surfaces. When the space is compact, it is uniquely ergodic, mixing of all orders, and exhibits rich connections with geometry, representation theory, spectral theory, and number theory. Beyond the classical homogeneous horocycle flow, a natural class of systems to study are its smooth time-changes: these are smooth perturbations which preserve the underlying orbits but modify the speed along them. 
Ratner's rigidity theorem \cite{Rat2} and Flaminio and Forni's solution of the cohomological equation \cite{FlaFo} imply that generic\footnote{i.e., in the complement of a subspace of countable codimension in a Sobolev space of sufficiently large order.} smooth time-changes are not measurably conjugate to the horocycle flow itself. As such, these flows provide a testing ground for understanding robustness and variability of parabolic behavior under nontrivial smooth perturbations. Their rigidity and fine ergodic properties have been the subject of intensive research, see, e.g., \cite{Rat2, Rat3, KLU, FlaFo2}.

A primary question in this context concerns mixing rates. While the exact rate of decay of correlations of H\"{o}lder observables was established by Ratner for the horocycle flow \cite{Rat1}, solving the same problem for general smooth parabolic flows remains notoriously difficult. The classical method of mixing-via-shearing provides upper bounds on correlation decay for sufficiently smooth observables and has been successfully applied in several settings: time-changes of horocycle and unipotent flows \cite{FU, Sim, Rav3}, time-changes of nilflows \cite{AFU, Rav5, AFRU, FK}, and smooth flows on surfaces \cite{Koc1, Koc2, SK, Ulc, Rav4, Fay1, Fay2, FFK, CW}.
Recently, Adam Kanigowski and the author have developed a general mixing-via-shearing argument that can be used to prove quantitative mixing of any order \cite{KR}.

On the other hand, except for the horocycle flow itself, no examples of smooth parabolic system are known where precise polynomial asymptotics hold, or even lower and upper bounds of the same order. In particular, for time-changes of horocycle flows on compact surfaces, Forni and Ulcigrai proved a polynomial upper bound for the correlations of smooth observables based on the equidistribution properties of sheared geodesic segments \cite{FU}. This result was later generalized by the author to time-changes of unipotent flows on finite volume quotients of semisimple Lie groups \cite{Rav3}. It is known, however, that, for the classical horocycle flow, arcs transverse to the weak-stable foliation for the geodesic flow (e.g., opposite horocycles or circle arcs) equidistribute at a faster rate than geodesic segments \cite{Rav2}. As a consequence, the precise quantitative behavior of time-changed horocycle flows has remained open.

In this paper, we address this problem for smooth time-changes of horocycle flows on compact hyperbolic surfaces. Our approach combines a refined version of the mixing-via-shearing method with a detailed description of ergodic integrals for horocycle flows, in particular of the regularity of the coefficients that appear in their asymptotic expansions. This combination allows us to establish polynomial bounds for the decay of correlations of smooth observables, matching the optimal rates by Ratner in the unperturbed setting \cite{Rat1}. Interestingly, in the presence of a spectral gap below $1/4$, we further obtain exact polynomial asymptotics, reproducing the classical behavior of the horocycle flow under time-changes. 

\medskip

In order to state our main result, we describe our setting and we fix some notation.
Let $M = \Gamma \backslash \SL_2(\R)$ be a compact quotient of $\SL_2(\R)$ by a discrete subgroup $\Gamma$. We identify elements $W \in \mathfrak{sl}_2(\R)$ of the Lie algebra $\mathfrak{sl}_2(\R)$ of $\SL_2(\R)$ with smooth vector fields on $M$; in particular, $W \in \mathfrak{sl}_2(\R) \setminus \{0\}$ induces the smooth flow on $M$ given by 
\[
x \mapsto x \exp(tW), \qquad \text{ for } t \in \R.
\]
We fix the basis of $\mathfrak{sl}_2(\R)$ given by 
\[
    X = 
    \begin{pmatrix}
    1/2 & 0 \\
    0 & -1/2
    \end{pmatrix},\qquad
    U = 
    \begin{pmatrix}
    0 & 1 \\
    0 & 0
    \end{pmatrix}, 
    \qquad
    V = 
    \begin{pmatrix}
    0 & 0 \\
    1 & 0
    \end{pmatrix}.
\]
The flows $\{\geo_t\}_{t\in \R}$ and $\{\horo_t\}_{t\in \R}$ corresponding to $X$ and $U$ above are called, respectively, the \emph{geodesic flow} and the (\emph{stable}) \emph{horocycle flow}.
We let $\widehat{X}, \widehat{U}, \widehat{V}$ denote the smooth 1-forms dual to the vector fields $X,U,V$. The volume form
\[
    \vol = \widehat{X} \wedge \widehat{U} \wedge \widehat{V}
\]
defines a probability measure, still denoted by $\vol$, on $M$ which is the unique probability measure invariant by  $\{\horo_t\}_{t\in \R}$. Locally, it is a Haar measure on $\SL_2(\R)$.

Let
\[
    Y = \frac{1}{2}(U+V), \qquad \Theta = \frac{1}{2}(U-V),
\]
and let $\square = - X^2 - Y^2 + \Theta^2 = -X^2 +X - UV$ be the Casimir operator. 
We let $\mu_0$ denote the smallest positive eigenvalue of $\square$ on $M$. We define
\[
    \nu_0 = \begin{cases} 
    \sqrt{1-4\mu} & \text{ if } \mu_0\leq 1/4, \\
    0 & \text{ otherwise},
    \end{cases}
    \qquad \text{and} \qquad
    R(t) := \begin{cases} 
    t^{\frac{1+\nu_0}{2}} & \text{ if } \mu_0 \neq 1/4, \\
    t^{\frac{1}{2}} \log t  & \text{ if } \mu_0 = 1/4. \\
    \end{cases}
\]
Finally, let $\Delta = \square - 2 \Theta^2 = -(X^2 + Y^2 + \Theta^2)$ and, for any $r\geq 0$, define the Sobolev space $W^r(M)$ to be the completion of $\mathscr{C}^{\infty}(M)$ with respect to the norm
\[
    \|f\|_{W^r} := \|(\Id + \Delta)^{r/2} f\|_2.
\]

In this paper, we are concerned with smooth time-changes of  $\{\horo_t\}_{t\in \R}$, which are defined as follows. Let $\alpha$ be a sufficiently smooth positive function on $M$, normalized so that $\vol(\alpha)=1$. 
The time-change induced by $\alpha$ is the flow $\{\tc_t\}_{t\in \R}$ whose generating vector field is $\Ual = \frac{1}{\alpha} U$. It preserves the smooth measure $\volal = \alpha \, \vol$.

The following is the quantitative mixing result proved by Forni and Ulcigrai \cite{FU}.

\begin{theorem}[{\cite[Theorem 19]{FU}}]\label{thm:FU_mixing}
Let $\alpha \in W^6(M)$ be a positive function, and let $\{\tc_t\}_{t\in\R}$ be the associated time-change of the horocycle flow on $M$. There exists a constant $\Cal>0$ such that, for any zero-average function $f\in W^6(M)$, for any $\ell \in W^1(M)$, and for any $t\geq 2$, we have
\[
    \left\lvert \int_M f\circ \tc_t \, \ell \volal \right\rvert \leq \Cal \|f\|_{W^6} (\|\ell\|_2^2 + \|X\ell\|_2^2)^{1/2}\frac{R(t)}{t}.
\]
\end{theorem}

In particular, \Cref{thm:FU_mixing} implies that the rate of decay of correlations for the time-change $\{\tc_t\}_{t\in\R}$ is bounded above by $t^{-\frac{1}{2}}$, if $\mu_0>1/4$, and by $t^{-\frac{1-\nu_0}{2}}$, if $\mu_0<1/4$. However, as we mentioned above, these rates are not optimal, since they do not match the rates of decay of correlations for the standard horocycle flow \cite{Rat1}.
Our main result proves that the rate of decay of correlations of sufficiently smooth time-changes is the same rate as for the horocycle flow itself (apart from an extra factor $\log t$ in the case where $\mu_0$ is exactly $1/4$). More precisely, we prove the following.

\begin{theorem}\label{thm:main}
There exists $\delta_0 >0$ such that the following holds. Let $\alpha \in W^7(M)$ be a positive function, and let $\{\tc_t\}_{t\in\R}$ be the associated time-change of the horocycle flow on $M$. There exists a constant $\Cal>0$ and a bilinear form $\cC_{\alpha}$ on $W^7(M)$ such that, for any zero-average functions $f, \ell \in W^7(M)$ and for any $t\geq 2$, we have the following estimates: if $\mu_0 > 1/4$, then 
\[
    \left\lvert \int_M f\circ \tc_t \, \ell \volal \right\rvert \leq \Cal \|f\|_{W^7} \|\ell\|_{W^7} t^{-1};
\]
if $\mu_0 = 1/4$, then
\[
    \left\lvert \int_M f\circ \tc_t \, \ell \volal \right\rvert \leq \Cal \|f\|_{W^7} \|\ell\|_{W^7} t^{-1}(\log t)^2;
\]
and, if $\mu_0 < 1/4$, then 
\[
    \left\lvert \int_M f\circ \tc_t \, \ell \volal - t^{-1 + \nu_0} \cC_{\alpha}(f,\ell)\right\rvert \leq \Cal \|f\|_{W^7} \|\ell\|_{W^7} t^{-1 + \nu_0 - \delta_0}.
\]
\end{theorem}

We remark that a precise asymptotic as above holds only when $\mu_0 < 1/4$ also for the classical horocycle flow. Indeed, when $\mu_0 > 1/4$, one can rescale the correlations of smooth observables by $t$ and show that they converge to a quasi-periodic motion on an infinite dimensional torus. A similar phenomenon happens if one considers the ergodic averages of smooth functions, see \cite{FlaFo, BF, Rav1}, and is related to the fact that there are countably many resonances for the geodesic flow on the line $\Re z = -1/2$.
In any case, we can show that, even in the case $\mu_0 > 1/4$, one cannot expect to improve the upper bound in \Cref{thm:main} and the correlations decay as $1/t$: the following result proves the continuity of the correlations in terms of the time-change function.
\begin{proposition}\label{thm:continuity}
Assume that $\mu_0>1/4$. For any $a > 1$, there exists a constant $C_a >1$ such that the following holds. Let $\alpha \in W^7(M)$ be a positive function so that $\|\alpha\|_{W^7} \leq C_a$ and $\|\alpha^{-1}\|_{W^7} \leq C_a$; denote by $\{\tc_t\}_{t\in \R}$ the associated time-change. Let $f,\ell \in W^7(M)$ be such that $\alpha f$ and $\alpha \ell$ are supported on the principal series. For any $t\geq 2$, we have
\[
    \Bigg\lvert    \int_M f\circ \tc_t \, \ell \volal -  \int_M  (\alpha f) \circ \horo_r \cdot \alpha \ell \, \vol \Bigg\rvert \leq \frac{C_a}{t} \, \nal^{\frac{1}{2}} \, \|f\|_{W^7} \, \|\ell\|_{W^7}.
\]
\end{proposition}

\medskip

\Cref{thm:main} readily implies that, whenever $\mu_0 \geq 1/4$, the spectral measures associated to zero-average functions $f\in W^7(M)$ are absolutely continuous with square-integrable Radon-Nykodym derivative. A natural question is whether the latter is uniformly bounded, at least away from 0. 
A softer argument than the one used in \Cref{thm:main}, which is actually analogous to the one used in \cite{FU}, shows that the correlations of \emph{coboundaries}, as a function of $t$, are absolutely integrable, independently of $\mu_0$. The estimate in \Cref{thm:main_spectral} below is (conjecturally) not optimal\footnote{One would expect the correlations of coboundaries to decay with order $O(t^{-2+\nu_0})$, with an additional factor $\log t$ if $\mu_0=1/4$.}; nonetheless, following \cite[Section 6]{FU}, it is sufficient to  strengthen \cite[Theorem 21]{FU} on the local behaviour of the spectral measures, proving that the spectral measures have a bounded density away from 0.

\begin{theorem}\label{thm:main_spectral}
Let $\alpha \in W^6(M)$ be a positive function, and let $\{\tc_t\}_{t\in\R}$ be the associated time-change of the horocycle flow on $M$. There exists a constant $\Cal>0$ such that, for any zero-average functions $f, \ell \in W^6(M)$ and for any $t\geq 2$, we have 
\[
    \Bigg\lvert\int_M \Ual f \circ \tc_t \cdot \ell \volal\Bigg\rvert\ll \|f\|_{W^6} \, \|\ell\|_{W^6} \frac{R(t)}{t^2}.
\]
Consequently, let $f \in W^6(M)$ be a zero average function, and let $\sigma_f$ be the associated spectral measure. For any $\xi \neq 0$ and for any $\varepsilon \in (0, |\xi|/2)$, we have 
\[
    |\sigma_f(\xi - \varepsilon, \xi + \varepsilon)| \leq \Cal \|f\|_{W^6} \frac{\varepsilon}{\xi^2}.
\]
\end{theorem}

\smallskip 

One crucial point in the classical mixing-via-shearing argument is to control the \emph{distortion} of the sheared segments, namely to quantify how much the growth of the push-forward of the transverse vector field (in this case, the geodesic flow) deviates from being linear. In this paper, the relevant quantity is $A^0_t(x)$, which can be expressed as the ergodic integral of the function $\frac{X\alpha}{\alpha}$. The corresponding bound is the content of \Cref{prop:correlations_with_integral}. Since the result might be interesting in its own right, we prove it in general, as follows.

\begin{theorem}\label{thm:main_2}
There exists $\delta_0 >0$ such that the following holds. Let $\{\tc_t\}_{t\in\R}$ be a smooth time-change of the horocycle flow on $M$.
Let $u,f,\ell \in W^7(M)$ with $\volal(u)=\volal(f)=\volal(\ell)=0$. For any $t\geq 2$, we have 
\[
    \Bigg\lvert \int_M \Big( \int_0^t u \circ \tc_r \diff r\Big) \cdot f\circ \tc_t \cdot \ell \volal \Bigg\rvert 
   \leq \Cal \|u\|_{W^7} \, \|f\|_{W^7} \, \|\ell\|_{W^7} \, E(t), 
\]
where
\[
   E(t) = 
   \begin{cases}
       \frac{R(t)^2}{t}, & \text{ if } \mu_0 \geq 1/4,\\
       t^{\nu_0-\delta_0}, & \text{ if } \mu_0 <1/4.
   \end{cases}
\]
\end{theorem}

\medskip

\paragraph{\textbf{Notation}}
In order to avoid cumbersome notation, in this paper we will use the following conventions.
\begin{itemize}
    \item The symbol $\delta_0$ denotes a positive constant, which might change from line to line, that depends only on the space $M$; notably, on the spectrum of $\square$. 
    \item When writing $A\ll B$, we mean that there exists a constant $C_{\alpha}$ so that $A \leq C_{\alpha} \cdot B$. The constant $C_{\alpha}$, which might change from line to line, depends only on the space $M$ and on the $W^7$-norm of the function $\alpha$ inducing the time-change, and on the $W^7$-norm of $\alpha^{-1}$.
    \item The notation $t^{\frac{1}{2}+}$ and $t^{\frac{1}{2}-}$ stands for $t^{\frac{1}{2}}\log t$ and $t^{\frac{1}{2}}$ respectively. 
\end{itemize}

\section{Preliminaries}

\subsection{Time-changes}

Let us fix a positive function $\alpha \in W^7(M)$, with $\vol(\alpha) = 1$, and let $\{\tc_t\}_{t\in \R}$ be the associated time-change of $\{\horo_t\}_{t\in \R}$.
Since the orbits of $\{\tc_t\}_{t\in \R}$ and $\{\horo_t\}_{t\in \R}$ coincide, there exists a function $\tau \colon M \times \R \to \R$ defined by the equality $\tc_t(x) = \horo_{\tau(x,t)}(x)$. Then, differentiating both sides, we have 
\begin{equation}\label{eq:alpha_tau}
    \frac{1}{\alpha} \circ \tc_t(x) = \frac{\partial \tau}{\partial t}(x,t),
\end{equation}
and thus, since $\tau(x,0)=0$, we can write
\begin{equation}\label{eq:definition_tau}
    \tau(x,t) = \int_0^t\frac{1}{\alpha} \circ \tc_r(x) \diff r, \qquad \text{for every $x\in M$ and $t\in \R$}.
\end{equation}
Reversing the roles of the vector fields $U$ and $\Ual$, if we define $\eta(x,t)$ by $\horo_t(x)=\tc_{\eta(x,t)}(x)$, the same reasoning as above tells us that
\[
    \eta(x,t) = \int_0^t \alpha \circ \horo_r(x) \diff r.
\]
In particular, since $t = \eta(x, \tau(x,t))$, it follows that
\begin{equation}\label{eq:t_tau}
    t = \int_0^{\tau(x,t)}\alpha \circ \horo_r(x) \diff r.
\end{equation}
As Equation \eqref{eq:t_tau} suggests, a precise understanding of ergodic integrals of the horocycle flow is crucial for proving our main theorem. 
Standard results on the deviation of ergodic averages imply that, for every $x \in M$ and $t \geq 2$, we have
\begin{equation}\label{eq:t_tau_bound}
   |\tau(x, \pm t) \mp t| \ll \nal \, R(t).
\end{equation}

As a matter of fact, we will need to use a much stronger result on the asymptotic expansion of ergodic integrals.
We recall in the next subsection the theorem we need.

We collect below some estimates on $\tau(x,t)$ we will need in this paper. For the sake of notation, we define
\[
    \tau_{\pm}(x,t) = \pm \tau(x,\pm t).
\]

\begin{lemma}\label{lem:Xtau_over_tau}
For every $x\in M$ and for every $t \geq 2$, we have
\[
    \left\lvert \frac{X\tau_{\pm}(x, t)}{\tau_{\pm}(x, t)} +1- \frac{1}{\alpha \circ \tc_{\pm t}(x)}\right\rvert \ll \nal \, \frac{R(t)}{t}.
\]
Furthermore, for every $r\in\R$, we have
\[
    \big\lvert X[\tau_{\pm}(\horo_rx, t)]\big\rvert \ll \nal \, (t+|r|).
\]
\end{lemma}
\begin{proof}
    For any $t\in \R$, differentiating both sides of \eqref{eq:t_tau} gives us
    \[
        \begin{split}
        0 &= X\tau(x,t) \cdot \alpha \circ \tc_t(x) +  \int_0^{\tau(x,t)}X(\alpha \circ \horo_r(x)) \diff r \\
        &= X\tau(x,t) \cdot \alpha \circ \tc_t(x) +  \int_0^{\tau(x,t)}X\alpha \circ \horo_r(x) + rU\alpha \circ \horo_r(x) \diff r \\
        &=X\tau(x,t) \cdot \alpha \circ \tc_t(x) + \tau(x,t) \cdot \alpha \circ \tc_t(x) - t +  \int_0^{\tau(x,t)}X\alpha \circ \horo_r(x)\diff r.
        \end{split}
    \]
    Let now $t\geq 2$.
    Since $\alpha \gg 1$ and since, by \eqref{eq:t_tau_bound}, we have $t \ll \tau_{\pm}(x, t)$, it follows that 
    \[
        \left\lvert \frac{X\tau_{\pm}(x, t)}{\tau_{\pm}(x, t)} +1- \frac{t}{\tau_{\pm}(x, t) \cdot \alpha \circ \tc_{\pm t}(x)}\right\rvert \ll \|X\alpha\|_{W^6} \frac{R(t)}{t}.
    \]
    The first claim follows from the fact that $|t/\tau_{\pm}(x,t) -1|\ll \nal \, t^{-1}R(t)$.

    As for the second statement, for every $r\in \R$ we have
    \[
        X[\tau_{\pm}(\horo_rx, t)] = D\horo_r(X)\tau_{\pm}(\horo_rx, t) = X\tau_{\pm}(\horo_rx, t) + r U\tau_{\pm}(\horo_rx, t).
    \]
    The first part of this lemma implies that $\|X\tau_{\pm}(\cdot, t)\|_{\infty}\ll \nal \, \|\tau_{\pm}(\cdot, t)\|_{\infty}\ll \nal \, t$. 
    On the other hand, \eqref{eq:definition_tau} yields
    \[
        U\tau_{\pm}(y, t) = \alpha(y) \Ual \tau_{\pm}(y, t) = {\pm}\alpha(y) \left( \frac{1}{\alpha \circ \tc_{\pm t}(y) } -  \frac{1}{\alpha (y)}\right),
    \]
    which gives us $\|U\tau_{\pm}(\cdot, t)\|_{\infty}\ll \nal$, and completes the proof.
\end{proof}

\begin{lemma}\label{lem:tau_beta_minus_t_beta}
For every $t \geq 2$ and $x\in M$, we have
\[
    |\log \big( \tau_{\pm}(x, t) \big)-\log t| \ll \nal \, \frac{R(t)}{t}.
\]
Moreover, for any $b\in \C$ with $\Re( b) \in [0,1]$ we have 
\[
    |\tau_{\pm}(x, t)^b -t^b| \ll \nal \, |b| \, t^{\Re (b) }\frac{R(t)}{t}.
\]
In particular, $|R(\tau_{\pm}(x,t)) - R(t)| \ll \nal \, t^{\frac{1+\nu_0}{2} - \delta_0}$.
\end{lemma}
\begin{proof}
    From the bound \eqref{eq:t_tau_bound}, we easily derive
    \[
        |\log \big( \tau_{\pm}(x, t)\big)-\log t| = \left\lvert \log \left(1+\frac{\tau_{\pm}(x, t) - t}{t} \right) \right\rvert \ll \left\lvert \frac{\tau_{\pm}(x, t) -t}{t} \right\rvert \ll \nal \, \frac{R(t)}{t},
    \]
    which proves the first statement. 
    
    Fix $b \in \C$. It is easy to see that 
    \[
        |\tau_{\pm}(x, t))^{b}-t^{b}| \leq |\tau_{\pm}(x,t) - t| \, |b| \, \min(\tau_{\pm}(x,t),t)^{\Re (b)-1},
    \]
    which proves the second claim.
    This also imply the bound on $|R(\tau_{\pm}(x,t)) - R(t)|$ in the case $\mu_0\neq 1/4$ from the definition of $R$, choosing $b = \frac{1+\nu_0}{2}$. For $\mu_0 = 1/4$, we use the fact that $ \tau_{\pm}(x, t) \ll t$ and the previous claim to conclude
    \[
    \begin{split}
        |\tau_{\pm}(x, t)^{\frac{1}{2}}\log \big( \tau_{\pm}(x, t)\big)-t^{\frac{1}{2}}\log t| &\ll | \tau_{\pm}(x, t)^{\frac{1}{2}}-t^{\frac{1}{2}}| \, \log t + t^{\frac{1}{2}} \, |\log \big( \tau_{\pm}(x, t)\big)-\log t| \\
        &\ll \nal \, \big((\log t)^2+\log t \big).
    \end{split}
    \]
    The proof is complete.
\end{proof}

\subsection{Ergodic integrals for time-changes}

Let $W \in \{X,U,V\}$ and let $j \in \N$. 
If $f \in L^{\infty}(M)$ is such that $Wf, \dots, W^jf \in L^{\infty}(M)$, we define
\[
    \|f\|_{W,j} := \|f\|_{\infty} + \|Wf\|_{\infty} + \cdots 
    + \|W^jf\|_{\infty}.
\]
We will often write $\|f\|_{W}$ instead of $\|f\|_{W,1}$. 
If $f \in L^{\infty}(M)$ is H\"{o}lder continuous in direction $W$ of some exponent $\beta \in (0,1)$, we define 
\[
    \|f\|_{W,\beta} := \|f\|_{\infty} + \sup_{\substack{ x\in M \\ r \in \R\setminus \{0\} }} \frac{|f(x\exp(rW)) - f(x)|}{r^{\beta}}.
\]
If $f \in L^{\infty}(M)$ is H\"{o}lder continuous of some exponent $\beta \in (0,1)$, we let 
\[
    \|f\|_{\beta} := \|f\|_{\infty} + \sup_{\substack{ x, y\in M \\ x\neq y }} \frac{|f(x) - f(y)|}{\dist(x,y)},
\]
where $\dist(x,y)$ is the distance between $x$ and $y$.

The following theorem, in a slightly different form, was proved in \cite{Rav1}; we explain below the differences in the notation. 
Let $\Spec_{\geq 0} := \Spec(\square) \cap \R_{\geq 0}$ and $\Spec_{> 0} := \Spec(\square) \cap \R_{> 0}$ denote, respectively, the set of nonnegative and of positive eigenvalues of the Casimir operator $\square$ on $M$. For every $\mu \in \Spec_{\geq 0}$, define $\nu=\nu(\mu)$ to be the unique $\nu \in [0,1] \cup i\R_{>0}$ such that $1-\nu^2 = 4\mu$. 
Furthermore, let 
\[
\lambda_{\pm} = \lambda_{\pm}(\mu) = 
\begin{cases}
    \frac{1 \pm \nu}{2} & \text{if $\mu \neq 1/4$,} \\
    \frac{1}{2}\pm & \text{if $\mu = 1/4$.} \\
\end{cases}
\]

\begin{theorem}\label{thm:Rav}
Fix $j\in \N$, and let $f\in W^{5+j}(M)$, with $\vol(f)=0$. For every $\mu \in \Spec_{\geq 0}$, there exist bounded functions
$\cP_{\lambda_{\pm}}f, \cQ_{\lambda_{\pm}}f \colon M\times [1,\infty) \to \C$ such that the following holds. Define
\[
    \begin{split}
    &\cD_{\lambda_{\pm}}f(x,T) := \cP_{\lambda_{\pm}}f (x,T)- \cQ_{\lambda_{\pm}}f (\horo_1(x),T).
    \end{split}
\]
For every $x\in M$ and for every $T\geq 1$, we have
\[
    \begin{split}
    \Bigg\lvert \int_{0}^{T} f \circ \horo_r(x) \diff r & - \sum_{\substack{ \mu \in \Spec_{\geq 0}\\ \bullet \in \{+,-\}}}T^{\lambda_{\bullet}} \, \cD_{\lambda_{\bullet}}f (\geo_{\log T}(x),T)\Bigg\rvert \ll \|f\|_{W^6}.
    \end{split}
\]
The functions $\cP_{\lambda_{\bullet}}f$ and $\cQ_{\lambda_{\bullet}}f$ satisfy the following properties:
\begin{itemize}
    \item[(a)] for $\mu = 0$, we have $\cD_{\lambda_{+}}f \equiv 0$ and $\cP_{\lambda_{-}}f = \cQ_{\lambda_{-}}f$,
    \item[(b)] for all $\mu \in \Spec_{\geq 0}$, $\cP_{\lambda_{\bullet}}f$ and $\cQ_{\lambda_{\bullet}}f$ are differentiable in $T$ and 
    \[
    \sum_{\substack{ \mu \in \Spec_{\geq 0}\\ \bullet \in \{+,-\}}} \left\|\frac{\partial}{\partial T}\cP_{\lambda_{\bullet}}f(\cdot, T) \right\|_{\infty} + \left\|\frac{\partial}{\partial T}\cQ_{\lambda_{\bullet}}f(\cdot, T) \right\|_{\infty}\ll \|f\|_{W^6} T^{-1-\lambda_{\bullet}},
    \]
    \item[(c)] for every $\mu \in \Spec_{>0}$, $\cP_{\lambda_{\bullet}}f$ and $\cQ_{\lambda_{\bullet}}f$ are $j$-times differentiable in the $X$--direction, and
    \[
    \sum_{\substack{ \mu \in \Spec_{>0}\\ \bullet \in \{+,-\}}} \|\cP_{\lambda_{\bullet}}f(\cdot, T)\|_{X,j} + \|\cQ_{\lambda_{\bullet}}f(\cdot, T)\|_{X,j} \ll \|f\|_{W^{5+j}},
    \]
    \item[(d)] for every $\mu \in \Spec_{>0}$, $\cP_{\lambda_{\bullet}}f$ and $\cQ_{\lambda_{\bullet}}f$ are H\"{o}lder continuous in the $U$-direction and, for any fixed $\beta < \frac{1-\nu_0}{2}$, we have
    \[
    \sum_{\substack{ \mu \in \Spec_{>0}\\ \bullet \in \{+,-\}}} \|\cP_{\lambda_{\bullet}}f(\cdot, T) \|_{U,\beta} + \|\cQ_{\lambda_{\bullet}}f(\cdot, T)\|_{U,\beta} \ll \|f\|_{W^6}.
    \]
\end{itemize}
\end{theorem}
\begin{proof}
    Let $f \in W^{5+j}(M)$, and let us write $f = \sum_{\mu \in \Spec(\square)}f_\mu$, where $f_\mu$ is an eigenfunction of the Casimir operator $\square$ with eigenvalue $\mu$. By \cite[Theorem 1]{Rav1} and \cite[Lemma 15]{Rav1}, the contribution of all components corresponding to $\mu <0$ can be bounded uniformly in $T$, and hence can be absorbed into the error term.

    For the component corresponding to $\mu = 0$, we have
    \[
    \Bigg\lvert \int_0^Tf_{\mu}\circ \horo_r(x) \diff r - \int_0^{\log T} - (Vf_0 \circ \geo_{\log T - \xi} - Vf_0 \circ \geo_{- \xi} \circ \horo_1 \circ \geo_{\log T })(x)\diff \xi\Bigg\rvert  \ll \|f_0\|_{\mathscr{C}^1},
    \]
    see \cite[Section 5]{Rav1}, hence we define 
    \[
    \cP_{\lambda_{-}(0)}f(x,T) = \cQ_{\lambda_{-}(0)}f(x,T) = - \int_0^{\log T}Vf_0 \circ \geo_{- \xi} (x)\diff \xi.
    \]
    This proves (a).
    
    For all the other $\mu \in \Spec_{>0} \setminus \{1/4\}$, the results in \cite[Sections 3--4]{Rav1} show that
    \[
        \int_0^Tf_{\mu}\circ \horo_r(x) \diff r = \sum_{\bullet \in \{+,-\}} T^{\lambda_{\bullet}} ( \cP_{\lambda_{\bullet}}f (x,T)- \cQ_{\lambda_{\bullet}}f(\horo_1(x),T)),
    \]
    where
    \[
    \begin{split}
        &\cQ_{\lambda_{\bullet}}f(x,T) = \frac{1}{\bullet \nu} \int_0^{\log T} e^{-\lambda_{\bullet}\xi}Vf_{\mu} \circ \geo_{- \xi} (x) \diff \xi, \qquad \text{and}\\
        &\cP_{\lambda_{\bullet}}f(x,T) = \cQ_{\lambda_{\bullet}}f(x,T) + \frac{1\bullet \nu}{\bullet 2\nu} \int_0^1 f_{\mu}\circ \horo_r(x)\diff r - \frac{1}{\bullet\nu} \int_0^1 Xf_{\mu}\circ \horo_r(x)\diff r.
    \end{split}
    \]
    Similar formulas hold in the case $\mu = 1/4$, see \cite[Section 4.2]{Rav1}.

    Since
    \[
    \frac{\partial}{\partial T}\cP_{\lambda_{\bullet}}f(x, T) = \frac{\partial}{\partial T}\cQ_{\lambda_{\bullet}}f(x, T) = \bullet \frac{1}{\nu T} \, T^{-\lambda_{\bullet}} Vf_{\mu} \circ \geo_{-\log T}(x),
    \]
    the claim (b) follows from the Sobolev Embedding Theorem as in \cite[Lemma 15]{Rav1}.

    Parts (c) is immediate using \cite[Lemma 15]{Rav1} and (d) is proved in   \cite[Section 4.4]{Rav1}.
\end{proof}

\begin{remark}\label{rmk:Rav}
    From the fact that $\int_0^{-T} f \circ \horo_r(x) \diff r = - \int_0^{T} f \circ \horo_r(\horo_{-T}x) \diff r$, it follows that 
    \[
    \begin{split}
        \Bigg\lvert \int_{0}^{-T} f \circ \horo_r(x) \diff r - \sum_{\substack{ \mu \in \Spec_{\geq 0}\\ \bullet \in \{+,-\}}}T^{\lambda_{\bullet}} \, (-\cD_{\lambda_{\bullet}}f (\horo_{-1} \circ \geo_{\log T}(x),T))\Bigg\rvert \ll \|f\|_{W^6},
    \end{split}
    \]
    for all $x\in M$ and $T\geq 2$. We will use the shorthand notation
    \[
    \cD_{\lambda_{\bullet}}^{+}f (x,T) = \cD_{\lambda_{\bullet}}f (x,T), \qquad \text{and} \qquad \cD_{\lambda_{\bullet}}^{-}f (x,T) = -\cD_{\lambda_{\bullet}}f (\horo_{-1}(x),T).
    \]
\end{remark}
From \Cref{thm:Rav} and \Cref{rmk:Rav}, one can easily recover the upper bound
\[
    \left\lvert \int_{0}^{\pm T} f \circ \horo_r(x) \diff r\right\rvert \ll \|f\|_{W^6} \, R(T),
\] 
for every $T\geq 2$.

\Cref{thm:Rav} can be used to study ergodic integrals for time-changes, since, for any $f \in W^6(M)$ and any $x\in M$, by \eqref{eq:alpha_tau} we have
\begin{equation}\label{eq:change_integrals}
    \int_0^t f \circ \tc_r(x) \diff r = \int_0^t (\alpha f) \circ \horo_{\tau(x,r)}(x) \cdot \frac{\partial \tau}{\partial r}(x,r) \diff r = \int_0^{\tau(x,t)} (\alpha f) \circ \horo_{r}(x) \diff r. 
\end{equation}
Therefore, for every $t\geq 2$, we have
\begin{equation*}
    \Bigg\lvert\int_0^{\pm t} f \circ \tc_r(x) \diff r - \sum_{\substack{ \mu \in \Spec_{\geq 0}\\ \bullet \in \{+,-\}}}\tau_{\pm}(x,t)^{\lambda_{\bullet}} \, \cD_{\lambda_{\bullet}}^{\pm}(\alpha f) (\geo_{\log \tau_{\pm}(x,t)}(x),\tau_{\pm}(x,t)) \Bigg\rvert \ll \|f\|_{W^6}. 
\end{equation*}
We note that, by \Cref{thm:Rav}-(b), we have 
\[
\sum_{\substack{ \mu \in \Spec_{\geq 0}\\ \bullet \in \{+,-\}}} \big|\tau_{\pm}(x,t)^{\lambda_{\bullet}} \, \cD_{\lambda_{\bullet}}^{\pm}(\alpha f) (\geo_{\log \tau_{\pm}(x,t)}(x),\tau_{\pm}(x,t)) - \tau_{\pm}(x,t)^{\lambda_{\bullet}} \, \cD_{\lambda_{\bullet}}^{\pm}(\alpha f) (\geo_{\log \tau_{\pm}(x,t)}(x),t)\big| \ll \frac{R(t)}{t} \|f\|_{W^6}, 
\]
thus we can conclude
\begin{equation}\label{eq:integral_horo_tc}
    \Bigg\lvert\int_0^{\pm t} f \circ \tc_r(x) \diff r - \sum_{\substack{ \mu \in \Spec_{\geq 0}\\ \bullet \in \{+,-\}}}\tau_{\pm}(x,t)^{\lambda_{\bullet}} \, \cD_{\lambda_{\bullet}}^{\pm}(\alpha f) (\geo_{\log \tau_{\pm}(x,t)}(x),t) \Bigg\rvert \ll \|f\|_{W^6}. 
\end{equation}

\begin{remark}\label{rmk:principal_series}
    If $\alpha f$ is supported only on the principal series, then we have 
    \begin{equation*}
    \Bigg\lvert\int_0^{\pm t} f \circ \tc_r(x) \diff r - \sum_{\substack{ \mu \in \Spec(\square) \cap (1/4,\infty)\\ \bullet \in \{+,-\}}}\tau_{\pm}(x,t)^{\lambda_{\bullet}} \, \cD_{\lambda_{\bullet}}^{\pm}(\alpha f) (\geo_{\log \tau_{\pm}(x,t)}(x),t) \Bigg\rvert \ll \frac{R(t)}{t}  \|f\|_{W^6}. 
\end{equation*}
\end{remark}

\subsection{Push-forward of geodesic arcs}

We are interested in the push-forward of the vector field $X$ by the action of the flow $\{\tc_t\}_{t\in \R}$. 
At the heart of the mixing-via-shearing method lies the fact that geodesic arcs get sheared along the flow direction. More precisely, we define
\[
    A^0_t(x) = \int_0^t \frac{X\alpha}{\alpha} \circ \tc_r \diff r, \qquad \text{and} \qquad A_t(x) = \int_0^t\left(1- \frac{X\alpha}{\alpha}\right) \circ \tc_r \diff r = t-A_t^0(x).
\]
Then, it is possible to prove that
\begin{equation}\label{eq:push-forward_X}
	D\tc_t(X) = X + A_t(x) \Ual,
\end{equation}
see, e.g., \cite[Section 3.1]{Sim}\footnote{See also \cite[Lemma 1]{FU}, up to sign.}. 
The quantity $ A^0_t(x)$ represents the distortion; namely, the deviation of the shear from a linear function.
From the results in the previous subsection, we can deduce the asymptotics behaviour of $A^0_{\pm t}$. 
If $\mu_0 \geq 1/4$, then, for all $t \geq 2$, we have
\[
 \|A^0_{\pm t}(x)\|_{\infty} \ll \xnal \, R(t). 
\]
Otherwise, if $\mu_0 < 1/4$, we let $\lambda_0 = \lambda_{+}(\mu_0) = \frac{1+\nu_0}{2}$ and we define 
\[
    \cA^{\pm}_t(x) := \cD_{\lambda_0}^{\pm}(X\alpha)(x,t).
\]
\Cref{thm:Rav} and the subsequent discussion imply that 
there exists\footnote{As previously mentioned, $\delta_0$ depends only on $M$ (namely, on the set $\Spec_{> 0}$).} $\delta_0 >0$ so that, for all $t\geq 2$, 
\[
    \|A^0_{\pm t}(x) - \tau_{\pm}(x, t)^{\lambda_0 } \, \cA_t^{\pm} \circ \geo_{\log \tau_{\pm}(x, t)}\|_{\infty} \ll t^{\lambda_0  - \delta_0}.
\]
By \Cref{lem:tau_beta_minus_t_beta} and the H\"{o}lder property of $\cA_t^{\pm}$, for every $x \in M$, we have that 
\[
\begin{split}
    &|\tau_{\pm}(x, t)^{\lambda_0}-t^{\lambda_0}| \ll t^{\lambda_0^2-1} \ll  t^{\lambda_0 - \delta_0},\\
    &| \cA_t^{\pm} \circ \geo_{\log \tau_{\pm}(x, t)} (x) - \cA_t^{\pm} \circ \geo_{\log t} (x) | \ll | \log (\tau_{\pm}(x, t)) - \log t|^{\frac{1-\nu_0}{2}} \ll t^{-\frac{(1-\nu_0)^2}{4}}\ll t^{-\delta_0}.
\end{split}
\]
We can thus conclude that, for all $t\geq 2$, 
\begin{equation}\label{eq:bound_on_A0t}
\begin{split}
    &\|A^0_{\pm t}(x)\|_{\infty} \ll \nal \, R(t) \qquad \text{ if $\mu_0 \geq 1/4$,}\\
    &\|A^0_{\pm t}(x) - t^{\lambda_0} \, \cA_t^{\pm} \circ \geo_{\log t}\|_{\infty} \ll t^{\lambda_0 - \delta_0} \qquad \text{ if $\mu_0 < 1/4$.}
\end{split}
\end{equation}

Clearly, we can replace $A^0_{\pm t}(x)$ with the ergodic integral of any other $w \in W^6(M)$ with $\volal(w)=0$ and obtain the same conclusion, as stated in the next lemma.
\begin{lemma}\label{lem:Holder_W}
Let $w \in W^6(M)$ with $\volal(w)=0$. 
If $\mu_0 \geq 1/4$, then for every $x \in M$ and $t\geq 2$, we have
\[
    \left\lvert \int_{0}^{\pm t} w \circ \tc_r(x) \diff r \right\rvert \ll \|w\|_{W^6} \, R(t).
\]
Otherwise, if $\mu_0 < 1/4$, there exist H\"{o}lder continuous functions $\cW_t^{\pm}$ of exponent $\beta = \frac{1-\nu_0}{2}$, satisfying $\|\cW_t^{\pm}\|_{\beta} \ll \|w\|_{W^6}$, for which the following holds. For any $x \in M$ and $t\geq 2$, we have 
\[
\begin{split}
    &\left\lvert \int_{0}^{\pm t} w \circ \tc_r(x) \diff r - t^{\lambda_0} \cW_t^{\pm} \circ \geo_{\log t}(x)\right\rvert \ll \|w\|_{W^6} \, t^{\lambda_0 - \delta_0}.
\end{split}
\] 
\end{lemma}
We conclude this section with some further estimates on $A_t$. Let $u$ be any $\mathscr{C}^1$ function on $M$. Then, by \eqref{eq:push-forward_X}, we have
\begin{equation}\label{eq:X_deriv_u}
\begin{split}
    X\Big(\int_0^t u \circ \tc_r \diff r\Big) &= \int_0^t X\big(u \circ \tc_r \big) \diff r = \int_0^t A_r \cdot \Ual u \circ \tc_r \diff r + \int_0^t X u \circ \tc_r \diff r \\
    &= A_t \cdot u \circ \tc_t - \int_0^t \frac{\diff}{\diff r} A_r \cdot u \circ \tc_r \diff r + \int_0^t X u \circ \tc_r \diff r  \\
    &= A_t \cdot u \circ \tc_t - \int_0^t \Big( u- \frac{X\alpha}{\alpha} u - Xu \Big) \circ \tc_r \diff r.
\end{split}
\end{equation}
Thus, we obtain the following lemma, as in \cite[Lemma 7]{FU}.

\begin{lemma}\label{lem:estimate_At}
Let $u \in \mathscr{C}^1(M)$. 
For all $t \geq 2$ and all $x\in M$, we have
\[
    \left\lvert X\Big(\int_0^t u \circ \tc_r(x) \diff r\Big)\right\rvert \ll t \|u\|_{\mathscr{C}^1}.
\]
In particular, $|XA_{\pm t}(x)|\ll t$.
\end{lemma}

\section{A refined mixing-via-shearing argument}

In this section, we prove our main theorem by developing a new version of the mixing-via-shearing argument. 
We start by proving some preliminary result.

\subsection{Variations on Forni and Ulcigrai's result}

First of all, we prove a couple of variations of Forni and Ulcigrai's result (i.e., of \Cref{thm:FU_mixing}) that are better suited for our purposes.
An interesting feature of their result, which will be crucial in our argument, is the \lq\lq asymmetry\rq\rq\ in the observables: while the bound in \Cref{thm:FU_mixing} involves the $\|\cdot\|_{W^6}$-norm of $f$, the dependence on the function $\ell$ is only in terms of the norm $\|\cdot\|_X$. We will make full advantage of this fact.

In \Cref{prop:improved_FU} below, we prove an upper bound on the correlations between two functions $w$ and $v$ which depends on a free parameter $\sigma \in (0,1)$. The idea will be to optimize the upper bound by choosing $\sigma$ appropriately, depending on the values of $\|v\|_{\infty}$ and of $\|v\|_{X}$.

\begin{proposition}\label{prop:improved_FU}
Let $w\in W^6(M)$, with $\volal(w)=0$, and let $v\in L^{\infty}(M)$ with $Xv\in  L^{\infty}(M)$. For every $\sigma \in (0,1)$ and for every $t\geq 2$, we have 
\begin{equation*}
    \left\lvert \int_M w\circ \tc_{\pm t} \, v \volal \right\rvert
    \ll \|w\|_{W^6} \Big( \|v\|_{\infty} \frac{R(\sigma t)}{\sigma t} + \|v\|_{X} \frac{R(\sigma t)}{t} \Big).
\end{equation*}
\end{proposition}
\begin{proof}
    The proof follows the argument developed in \cite{FU}. Using the invariance of the measure $\vol$ by $\geo_t$, for any fixed $\sigma \in (0,1)$, we have
    \begin{equation}\label{eq:propFU_1}
        \int_M w\circ \tc_{\pm t} \cdot v \volal =  \int_M w\circ \tc_{\pm t} \cdot A_{\pm t} \cdot \frac{\alpha v}{A_{\pm t}} \, \vol = \frac{1}{\sigma} \int_0^{\sigma} \int_M w\circ \tc_{\pm t} \circ \geo_s \cdot A_{\pm t} \circ\geo_s \cdot \frac{\alpha v}{A_{\pm t}} \circ\geo_s \, \vol \, \diff s.
    \end{equation}
    An integration by parts yields
    \begin{equation*}
    \begin{split}
        \int_M w\circ \tc_{\pm t} \cdot v \volal = & \frac{1}{\sigma} \int_M \left( \int_0^{\sigma} w\circ \tc_{\pm t} \circ \geo_s \cdot A_{\pm t} \circ\geo_s \, \diff s\right) \cdot \frac{\alpha v}{A_{\pm t}} \circ\geo_{\sigma} \, \vol \\
        &-  \frac{1}{\sigma} \int_0^{\sigma} \int_M  \left( \int_0^{s} w\circ \tc_{\pm t} \circ \geo_r \cdot A_{\pm t} \circ\geo_r \, \diff r \right) \cdot X\Big(\frac{\alpha v}{A_{\pm t}}\Big) \circ\geo_{s} \, \vol \diff s.
    \end{split}
    \end{equation*}
    By \Cref{lem:estimate_At}, we note that 
    \[
        \left\| X\Big(\frac{\alpha v}{A_{\pm t}}\Big) \right\|_{\infty} \ll \|v\|_X \, \frac{\|A_{\pm t}\|_{\infty}}{t^2} + \|v\|_{\infty} \, \frac{\|A_{\pm t}\|_X}{t^2} \ll \frac{\|v\|_X}{t}. 
    \]
    We call $\gamma_s^x$ the geodesic segment $\gamma_s^x \colon r \mapsto \geo_r(x)$ for $r\in[0,s]$.
    Then, we can write
    \[
        \left\lvert \int_M w\circ \tc_{\pm t} \, v \volal \right\rvert \ll \sup_{x\in M} \Bigg[ \frac{\|v\|_{\infty}}{\sigma t} \left\lvert  \int_{\tc_{\pm t} \gamma_{\sigma}^x} w \, \widehat{\Ual} \right\rvert  + \frac{\|v\|_{X}}{t} \sup_{s \in [0,\sigma]} \left\lvert  \int_{\tc_{\pm t} \gamma_{s}^x} w \, \widehat{\Ual} \right\rvert \Bigg],
    \]
    where $\widehat{\Ual} = \alpha \widehat{U}$ is the 1-form dual to $\Ual$.
    The conclusion follows from the bound
    \begin{equation}\label{eq:sheared_geodesics}
        \left\lvert  \int_{\tc_{\pm t} \gamma_{s}^x} w \, \widehat{\Ual}  \right\rvert = \left\lvert  \int_{\tc_{\pm t} \gamma_{\sigma}^x} \alpha w \, \widehat{U}  \right\rvert \ll \|w\|_{W^6} \, R(st) \ll \|w\|_{W^6} \, R(\sigma t).
    \end{equation}
    see \cite[Lemma 17]{FU}.
\end{proof}

In \Cref{prop:improved_FU_holder}, we obtain a modest power decay of the correlations of $w$ and $v$, but which holds in weaker regularity; namely, the function $v$ is required to be only H\"{o}lder continuous in the geodesic direction.

\begin{proposition}\label{prop:improved_FU_holder}
Let $w\in W^6(M)$, with $\volal(w)=0$, and let $v \in L^{\infty}(M)$ be H\"{o}lder continuous in direction $X$ of exponent $\beta$. There exists $\beta_0 \in (0,1)$ such that, for every $t \geq 2$, we have 
\begin{equation*}
    \left\lvert \int_M w\circ \tc_{\pm t} \, v \volal \right\rvert \ll \|w\|_{W^6} \|v\|_{X,\beta}  t^{-\beta_0}.
\end{equation*}
\end{proposition}
\begin{proof}
    For any $x\in M$ and for all $s \in [0,\sigma]$, \Cref{lem:estimate_At} yields
    \[
    \begin{split}
        \left\lvert \frac{(\alpha v)\circ\geo_s(x) }{A_{\pm t}\circ\geo_s(x)}  - \frac{(\alpha v)(x)}{A_{\pm t}(x)}\right\rvert &\ll \frac{\| (\alpha v) \circ \geo_s - (\alpha v)\|_{\infty} }{t} + \|v\|_{\infty} \frac{\| A_{\pm t} \circ \geo_s - A_{\pm t}\|_{\infty} }{t^2} \ll \frac{\sigma^{\beta}}{t} \|v\|_{X,\beta} + \frac{\sigma}{t} \|v\|_{\infty}\\
        &\ll \frac{\sigma^{\beta}}{t} \|v\|_{X,\beta}.
    \end{split}
    \]
    Hence, from \eqref{eq:propFU_1}, we obtain
    \[
        \left\lvert \int_M w\circ \tc_{\pm t} \, v \volal  -  \int_M \frac{1}{\sigma} \left(\int_0^{\sigma} w\circ \tc_{\pm t} \circ \geo_s \,  A_{\pm t} \circ\geo_s \,   \diff s\right)  \frac{\alpha v}{A_{\pm t}} \, \vol  \right\rvert \ll \sigma^{\beta} \|v\|_{X,\beta} \|w\|_{\infty}.
    \]
    By \eqref{eq:sheared_geodesics},
    \[
        \left\lvert \int_M w\circ \tc_{\pm t} \, v \volal  \right\rvert \ll (\sigma t)^{-\frac{1-\nu_0}{2}} \log (\sigma t) \|w\|_{W^6} \|v\|_{\infty}+\sigma^{\beta} \|v\|_{X,\beta} \|w\|_{\infty}.
    \]
    Choosing, for example, $\sigma = t^{-\frac{1}{2}}$ proves the result, with 
    $
        \beta_0 = \min \Big\{ \frac{\beta}{2}, \frac{1-\nu_0}{5} \Big\}.
    $
\end{proof}

\subsection{The key lemma}

The following lemma is the starting point of the mixing-via-shearing method: by the shearing properties of the flow on the geodesic vector field, see \Cref{eq:push-forward_X}, we can rewrite the correlations between two functions as an expression involving ergodic averages.

\begin{lemma}\label{lem:key_lemma}
Let $w,v \in L^{\infty}(M)$ be such that $Xw,Xv \in L^{\infty}(M)$. Then, for all $t \neq 0$, we have
\[
\begin{split}
    \int_M w\circ \tc_t \, v \volal = &\frac{1}{t}\int_M A^0_t \cdot w \circ \tc_t \cdot  v \volal  - \frac{1}{t} \int_M \Big( \int_0^{t} \frac{X(\alpha w)}{\alpha} \circ \tc_r \diff r \Big) \cdot v \volal \\
    &+ \frac{1}{t} \int_M \Big( \int_0^{t} w \circ \tc_r \diff r \Big) \cdot \Big[v - \frac{X(\alpha v)}{\alpha} \Big] \volal.
\end{split}
\]
\end{lemma}
\begin{proof}
    By assumption and by measure-invariance, for any $r\in\R$, we have
    \[
        \int_M w \circ \tc_r \cdot X(\alpha v) \vol = -\int_M X(w \circ \tc_r) \cdot \alpha v \vol.
    \]
    Since $X(w \circ \tc_r) = D\tc_r(X)w \circ \tc_r$, by \eqref{eq:push-forward_X}, it follows that 
    \[
        \int_M w \circ \tc_r \cdot X(\alpha v) \vol = -\int_M [Xw \circ \tc_r +A_r \cdot \Ual w \circ \tc_r ]\cdot (\alpha v) \vol.
    \]
    Rearranging the terms and integrating from 0 to $t$, we get
    \[
        \int_0^t \int_M A_r \cdot \Ual w \circ \tc_r \cdot \alpha v \, \vol \diff r = -\int_0^t \int_M \left(  w \circ \tc_r \cdot X(\alpha v)  + Xw \circ \tc_r \cdot \alpha v \right)\vol \diff r.
    \]
    We integrate by parts the left-hand side, which can be rewritten as
    \[
    \begin{split}
        &\int_M \Big(\int_0^t  A_r \cdot \frac{\diff}{\diff r}( w \circ \tc_r) \diff r \Big)  \cdot \alpha v \vol = \int_M A_t \cdot w \circ \tc_t \cdot \alpha v \vol -  \int_M \Big( \int_0^t \frac{\diff}{\diff r} A_r \cdot w \circ \tc_r \diff r\Big) \cdot \alpha v \vol \\
        & \qquad = t \int_M w \circ \tc_t \cdot v \volal - \int_M A^0_t \cdot w \circ \tc_t \cdot \alpha v \vol - \int_M \int_0^t \Big( w - \frac{X\alpha}{\alpha}w \Big) \circ \tc_r \cdot \alpha v \, \diff r \, \vol, 
    \end{split}
    \]
    Therefore,
    \[
    \begin{split}
        t \int_M w \circ \tc_t \cdot v \volal = &\int_M A^0_t \cdot w \circ \tc_t \cdot \alpha v \vol + \int_M \int_0^t \Big( w - \frac{X\alpha}{\alpha}w - Xw \Big) \circ \tc_r \cdot \alpha v \, \diff r \, \vol \\
        &- \int_M \int_0^t  w \circ \tc_r \cdot X(\alpha v) \, \diff r \, \vol.
    \end{split}
    \]
    By \eqref{eq:integral_horo_tc}, we conclude
    \[
    \begin{split}
        t \int_M w \circ \tc_t \cdot v \volal
        = &\int_M A^0_t \cdot w \circ \tc_t \cdot \alpha v \vol + \int_M \Big( \int_0^{t} \Big( w - \frac{X(\alpha w)}{\alpha} \Big) \circ \tc_r \diff r \Big) \cdot \alpha v \vol \\
        &- \int_M \Big( \int_0^{t} w \circ \tc_r \diff r\Big) \cdot X(\alpha v) \vol,
    \end{split}
    \]
    which 
    completes the proof.
\end{proof}

From \Cref{lem:key_lemma} above, we deduce a first rough estimate on the decay of correlations. This is essentially \Cref{thm:FU_mixing} by Forni and Ulcigrai; however, in the case $\mu_0<1/4$, we can extract a precise expression for the higher order term.

\begin{corollary}\label{cor:alternative_FU}
Let $w\in W^7(M)$, with $\volal(w) =0$. 
If $\mu_0 \geq 1/4$, then for every $v \in L^{\infty}(M)$ for which $Xv \in L^{\infty}(M)$, we have
\[
    \Bigg\lvert \int_M w\circ \tc_{\pm t} \, v \volal \Bigg\rvert    \ll \|w\|_{W^7} \, \|v\|_X \frac{R(t)}{t}. 
\]
If $\mu_0 < 1/4$, there exist H\"{o}lder continuous functions $\cW_t^{\pm}, \widehat{\cW}_t^{\pm}$ of exponent $\beta = \frac{1-\nu_0}{2}$, satisfying $\|\cW_t^{\pm}\|_{\beta} + \|\widehat{\cW}_t^{\pm}\|_{\beta} \ll \|w\|_{W^6}$, for which the following holds.
For any $v \in L^{\infty}(M)$ for which $Xv \in L^{\infty}(M)$, and for any $t\geq 2$, we have
\[
\begin{split}
   \Bigg\lvert \int_M w\circ \tc_{\pm t} \, v \volal \Bigg\rvert \ll t^{\lambda_0-1} \Bigg[
   &\Bigg\lvert \int_M w \circ \tc_{\pm t} \cdot v \cdot {\cA}_t^{\pm} \circ \geo_{\log t} \volal \Bigg\rvert  + \Bigg\lvert \int_M \widehat{\cW}_t^{\pm} \circ \geo_{\log t}  \cdot v \volal \Bigg\rvert \\
   &+ \Bigg\lvert \int_M \cW_t^{\pm} \circ \geo_{\log t}  \cdot \Big(v- \frac{X(\alpha v)}{\alpha} \Big) \volal \Bigg\rvert \Bigg] + \|w\|_{W^7} \, \|v\|_X t^{\lambda_0-1 -\delta_0}.
\end{split}
\]
\end{corollary}

\begin{proof}
    We divide the proof in two cases, depending on the value of $\mu_0$.
    
    \noindent \textbf{Case 1.}
    Let us first assume $\mu_0 \geq 1/4$. By \Cref{lem:key_lemma}, \Cref{lem:Holder_W}, and \eqref{eq:bound_on_A0t}, we have
    \[
    \begin{split}
        \Bigg\lvert \int_M w\circ \tc_{\pm t} \, v \volal \Bigg\rvert   & \ll \frac{1}{t}  \|A_t^0\|_{\infty} \, \|w\|_{\infty}\, \|v\|_{\infty} + \frac{R(t)}{t} \big( \|X(\alpha w)\|_{W^6} \|v\|_{\infty} + \|w\|_{W^6} \|v\|_{X} \big)  \\
        & \ll \frac{R(t)}{t}\|w\|_{W^7} \|v\|_{X},
    \end{split}
    \]
    which proves our claim.

    \medskip

    \noindent \textbf{Case 2.}
    Assume now $\mu_0 < 1/4$. 
    By \Cref{lem:Holder_W}, there exist H\"{o}lder functions $\cW_t^{\pm}$ and $\widehat{\cW}_t^{\pm}$ for which we have
    \[
    \begin{split}
        &\Bigg\lvert \int_0^{\pm t} w \circ \tc_r(x) \diff r -  t^{\lambda_0} \cW_t^{\pm} \circ \geo_{\log t}  \Bigg\rvert \ll t^{\lambda_0 - \delta_0} \, \|w\|_{W^6},\\
        &\Bigg\lvert \int_0^{\pm t} \frac{X(\alpha w )}{\alpha} \circ \tc_r(x) \diff r -  t^{\lambda_0} \widehat{\cW}_t^{\pm} \circ \geo_{\log t}  \Bigg\rvert \ll t^{\lambda_0 - \delta_0} \, \|w\|_{W^7}.
    \end{split}
    \]
    We substitute in the statement of \Cref{lem:key_lemma} and we get
    \begin{multline*}
        \Bigg\lvert \int_M w\circ \tc_{\pm t} \, v \volal - \frac{1}{t} \int_M A^0_{\pm t} \cdot w \circ \tc_t \cdot v \volal \Bigg\rvert \ll  t^{\lambda_0-1} \Bigg\lvert \int_M \widehat{\cW}_t^{\pm} \circ \geo_{\log t}  \cdot v \volal \Bigg\rvert \\
        +  t^{\lambda_0-1} \Bigg\lvert \int_M \cW_t^{\pm} \circ \geo_{\log t}  \cdot \Big(v- \frac{X(\alpha v)}{\alpha} \Big) \volal \Bigg\rvert + \|w\|_{W^7} \, \|v\|_X t^{\lambda_0-1 -\delta_0}.
    \end{multline*}
    Again, by  \eqref{eq:bound_on_A0t}, we have
    \[
        \Bigg\lvert\int_M A^0_{\pm t} \cdot w \circ \tc_{\pm t} \cdot v \volal - \int_M t^{\lambda_0} \cA_t^{\pm} \circ \geo_{\log t} \cdot w \circ \tc_{\pm t} \cdot v \volal \Bigg\rvert \ll t^{\lambda_0 - \delta_0} \, \|w\|_{\infty}\, \|v\|_{\infty},
    \]
    and hence the proof is complete.
\end{proof}

We also state another corollary of \Cref{lem:key_lemma} which we will use at the end of the proof of \Cref{thm:main}. 

\begin{corollary}\label{cor:Ratner}
Let $f,\ell \in W^7(M)$. For any $t\geq 2$, we have
\[
t\int_M f\circ \horo_t \, \ell \vol =  \int_0^{t} \int_M   \left[ (f - Xf ) \circ \horo_r \cdot  \ell - f \circ \horo_r  \cdot X\ell \right] \vol \diff r.
\]
Moreover, assume that $\mu_0<1/4$. There exists a bilinear form $\cC$ on $W^7(M)$ such that, for every $t\geq 2$, 
\[
\Bigg\lvert \int_0^{t} \int_M \left[ (f - Xf ) \circ \horo_r \cdot  \ell - f \circ \horo_r  \cdot X\ell \right] \vol \diff r - t^{\nu_0} \cC(f,\ell) \Bigg\rvert    \ll \|f\|_{W^7} \, \|\ell\|_{W^7}  t^{\nu_0 -\delta_0}. 
\]
\end{corollary}
\begin{proof}
%
The first claim follows immediately from \Cref{lem:key_lemma} applied to the trivial case $\alpha \equiv 1$, hence the second claim translates to the statement on the asymptotics of the correlations for the horocycle flow
\[
\Bigg\lvert \int_M f\circ \horo_t \, \ell \vol - t^{-1+\nu_0} \cC(f,\ell) \Bigg\rvert    \ll \|f\|_{W^7} \, \|\ell\|_{W^7} \,  t^{-1+ \nu_0 -\delta_0}. 
\]
which  can be found, e.g., in \cite[Corollary 2.5]{FlaRav}. 
\end{proof}

\subsection{Proof of \Cref{thm:main}}

We conclude this section by proving \Cref{thm:main}. We will use the following two results, which are the most technical steps.
\begin{proposition}\label{prop:correlations_with_integral}
Let $f,\ell \in W^7(M)$ with $\volal(f) = \volal(\ell)=0$. For any $t\geq 2$, we have
\begin{equation*}
   \Bigg\lvert \int_M A_t^0 \cdot f\circ \tc_t \cdot \ell \volal \Bigg\rvert 
   \ll 
   \begin{cases}
       \xnal \, \|f\|_{W^7} \, \|\ell\|_{W^7} \frac{R(t)^2}{t}, & \text{ if } \mu_0 \geq 1/4,\\
       \xnal \, \|f\|_{W^7} \, \|\ell\|_{W^7} \,t^{\nu_0-\delta_0}, & \text{ if } \mu_0 <1/4.
   \end{cases}
\end{equation*}
\end{proposition}
\Cref{prop:correlations_with_integral} above is an immediate consequence of \Cref{thm:main_2} with $u=\frac{X\alpha}{\alpha}$, which is proved in the next section. 
The second result that we need, whose proof is also postponed to the next section, is the following.
Fix $\theta = \nal \, t^{1-\delta_0} \geq 2 \|\tau(\cdot,t) - t\|_{\infty}$, and let us define $t_0 = t_0(t)$ by
\[
    t_0 := t-\theta.
\]
Clearly, by \eqref{eq:t_tau_bound}, we have $|\tau(x,t) - t_0| \ll \theta + \|\tau(\cdot,t) - t\|_{\infty} \ll \theta$ as well as $|\tau(x,t) - t_0| \gg \theta -\|\tau(\cdot,t) - t\|_{\infty}\gg \theta$ , for all $x \in M$.

\begin{proposition}\label{prop:correlations_with_small_integral}
Let $w,v \in W^6(M)$ with $\vol(w) = \vol(v)=0$.
We have
\[
    \left\lvert \int_M \int_{t_0}^{\tau(x,t)} w \circ \horo_r(x) \cdot v(x) \diff r \vol \right\rvert \ll \begin{cases}
       \|w\|_{W^6}\, \|v\|_{W^6}\, \theta^{\nu_0}, & \text{ if } \mu_0 \neq 1/4,\\
       \|w\|_{W^6}\, \|v\|_{W^6}\, \theta^{\nu_0} \, (\log t)^2, & \text{ if } \mu_0 = 1/4.
   \end{cases}
\]
Furthermore, if $\mu_0 > 1/4$ and if $w$ is supported only on the principal series, we have
\[
    \left\lvert \int_M \int_{t_0}^{\tau(x,t)} w \circ \horo_r(x) \cdot v(x) \diff r \vol \right\rvert \ll 
       \nal^{\frac{1}{2}}\, \|w\|_{W^6}\, \|v\|_{W^6}.
\]
\end{proposition}

Note that, up to changing $\delta_0$, when $\mu_0<1/4$ we have $\theta^{\nu_0} \ll t^{\nu_0-\delta_0}$. 
We are now ready to prove our main result. Let $f,\ell \in W^7(M)$ with $\volal(f) = \volal(\ell)=0$, and let $t\geq 2$. 
\Cref{lem:key_lemma} and \eqref{eq:change_integrals} yield
\[
\begin{split}
    &\int_M f\circ \tc_t \, \ell \volal = \frac{1}{t}\int_M A^0_t \cdot f \circ \tc_t \cdot \alpha \ell \vol \\
    &\qquad \qquad + \frac{1}{t} \int_M  \Big( \int_{t_0}^{\tau(x,t)} (\alpha f)  \circ \horo_r\diff r \Big) \cdot \left[ \alpha \ell - X(\alpha \ell) \right] \vol - \frac{1}{t} \int_M  \Big( \int_{t_0}^{\tau(x,t)} X(\alpha f)  \circ \horo_r\diff r \Big) \cdot (\alpha \ell) \vol\\
    &\qquad \qquad + \frac{1}{t} \int_0^{t_0} \Big[ \int_M  (\alpha f) \circ \horo_r \cdot \alpha \ell \, \vol - \int_M  X(\alpha f) \circ \horo_r \cdot \alpha \ell \, \vol  - \int_M (\alpha f) \circ \horo_r  \cdot X(\alpha \ell)  \vol \Big] \diff r.
\end{split}
\]

\noindent \textbf{Case 1.}
If $\mu_0 < 1/4$, then,  from \Cref{prop:correlations_with_integral} and \Cref{prop:correlations_with_small_integral}, we obtain
\begin{multline*}
    \Bigg\lvert    \int_M f\circ \tc_t \, \ell \volal -  \frac{1}{t} \int_0^{t_0} \Big[ \int_M  (\alpha f) \circ \horo_r \cdot \alpha \ell \, \vol - \int_M  X(\alpha f) \circ \horo_r \cdot \alpha \ell \, \vol  - \int_M (\alpha f) \circ \horo_r  \cdot X(\alpha \ell)  \vol \Big] \diff r \Bigg\rvert \\
    \ll \|f\|_{W^7} \|\ell\|_{W^7} t^{-1+\nu_0 -\delta_0}.
\end{multline*}
\Cref{cor:Ratner} implies that
\begin{equation*}
    \Bigg\lvert    \int_M f\circ \tc_t \, \ell \volal - \frac{t_0^{\nu_0} }{t} \cC(\alpha f, \alpha \ell) \Bigg\rvert  
    \ll \|f\|_{W^7} \|\ell\|_{W^7} t^{-1+\nu_0 -\delta_0}.
\end{equation*}
Defining $\cC_{\alpha}(f, \ell) := \cC(\alpha f, \alpha \ell)$ proves the result in this case.

\medskip

\noindent \textbf{Case 2.}
If $\mu_0=1/4$, using again \Cref{prop:correlations_with_integral}, \Cref{prop:correlations_with_small_integral}, and Ratner's effective mixing result \cite[Theorem 2]{Rat1}, we obtain 
\[
\begin{split}
    \Bigg\lvert    \int_M f\circ \tc_t \, \ell \volal \Bigg\rvert \ll  \|f\|_{W^7} \|\ell\|_{W^7} \Big(\frac{1}{t} \int_1^{t_0}  r^{-1} \log r \diff r + \frac{(\log t)^2}{t} \Big) \ll  \|f\|_{W^7} \|\ell\|_{W^7} \frac{(\log t)^2}{t}.
\end{split}
\]

\medskip

\noindent \textbf{Case 3.}
The proof in the case $\mu_0 > 1/4$ is analogous to the previous case and is immediate, once we establish the following lemma.

\begin{lemma}\label{lem:homogeneous_corr}
Assume $\mu_0 > 1/4$. For every $w,v \in W^7(M)$ with $\vol(w)=\vol(v) = 0$, we have
\[
    \Bigg\lvert  \frac{1}{t} \int_1^{t} \Big( \int_M w\circ \horo_r\, v \vol \Big) \diff r \Bigg\rvert \ll \|w\|_{W^7} \|v\|_{W^7} t^{-1}.
\]
\end{lemma}
\begin{proof}
    Let us first assume that $w=w_n$ and $v=v_m$ satisfy $\square w_n = \mu w_n$, $\square v_m = \mu v_m$ for some Casimir eigenvalue $\mu >1/4$, and $\Theta w_n = i n w_n$ and $\Theta v_m = i m v_m$. We can also assume that $\|w_n\|_2= \|v_m\|_2 =1$. Under these assumptions, the function 
    \[
        x(t) := \int_M w_n \circ \horo_t\, v_m \vol
    \]
    satisfies the differential equation
    \begin{equation}\label{eq:eqdiff_Ratner}
        t^2x''(t) +3tx'(t) +4\mu x(t) = y(t),
    \end{equation}
    where 
    \[
        y(t) = - 4 x''(t) - \frac{4}{t}x'(t) - 4 i (m+n) x'(t) - 2i \frac{m+n}{t}x(t) + 4\frac{(m-n)^2}{t^2}x(t),
    \]
    see, e.g., \cite[Proposition 2.3]{FlaRav}. 
    \emph{A posteriori}, following Ratner's Theorem \cite[Theorem 2]{Rat1}, we have the estimate $|x(t)|\ll \|w_n\|_{W^3} \|v_m\|_{W^3} t^{-1}$. By \cite[Equation (10)]{FlaRav}, from the latter we deduce $|x'(t)|\ll \|w_n\|_{W^3} \|v_m\|_{W^3} t^{-2}$, from which, by \eqref{eq:eqdiff_Ratner}, it follows that $|x''(t)|\ll \|w_n\|_{W^3} \|v_m\|_{W^3} t^{-3}$. Consequently, we can bound
    $|y(t)|\ll \|w_n\|_{W^3} \|v_m\|_{W^3} t^{-2}$. Combining these estimates with \eqref{eq:eqdiff_Ratner}, we deduce
    \[
        \Bigg\lvert  \int_1^t r^2x''(r) + 3 rx'(r) + 4\mu x(r) \diff r \Bigg\rvert \ll \|w_n\|_{W^3} \|v_m\|_{W^3}. 
    \]
    Integrating by parts the left-hand side above gives us
    \[
        (4\mu - 1) \Bigg\lvert \int_1^t x(r) \diff r \Bigg\rvert \ll t^2 |x'(t)| + t|x(t)| + \|w_n\|_{W^3} \|v_m\|_{W^3} \ll \|w_n\|_{W^3} \|v_m\|_{W^3}.
    \]
    We proved our claim under our additional assumptions on $w=w_n$ and $v=v_m$. 
    Standard results from harmonic analysis allow to express correlations of arbitrary functions in $W^7(M)$ as a countable sum of correlations of functions $w_n, v_m$ as above.
\end{proof}

Finally, let us turn to the proof of \Cref{thm:continuity}. Assume we are still in the case $\mu_0 > 1/4$, and also that $\alpha f$ and $\alpha \ell$ are supported on the principal series. The same argument as above, using the second part of \Cref{prop:correlations_with_small_integral}, proves that 
\begin{multline*}
    \Bigg\lvert    \int_M f\circ \tc_t \, \ell \volal -  \frac{1}{t} \int_0^{t_0} \Big[ \int_M  (\alpha f) \circ \horo_r \cdot \alpha \ell \, \vol - \int_M  X(\alpha f) \circ \horo_r \cdot \alpha \ell \, \vol  - \int_M (\alpha f) \circ \horo_r  \cdot X(\alpha \ell)  \vol \Big] \diff r \Bigg\rvert \\
    \ll \nal^{\frac{1}{2}} \, \|f\|_{W^7} \|\ell\|_{W^7} t^{-1}. 
\end{multline*}
By \Cref{cor:Ratner}, we have
\[
\Bigg\lvert    \int_M f\circ \tc_t \, \ell \volal -  \frac{t_0}{t} \int_M (\alpha f)\circ \horo_{t_0} \, \alpha \ell \vol \Bigg\rvert \\
    \ll \nal^{\frac{1}{2}} \, \|f\|_{W^7} \|\ell\|_{W^7} \, t^{-1}.
\]
Correlations of coboundaries for the horocycle flow decay as $O(t^{-2})$, see, e.g., the proof of \Cref{lem:homogeneous_corr} or of \Cref{thm:main_spectral}. Therefore,
\[
\begin{split}
    \Bigg\lvert \int_M (\alpha f)\circ \horo_{t} \, \alpha \ell \vol - \int_M (\alpha f)\circ \horo_{t_0} \, \alpha \ell \vol \Bigg\rvert &= \Bigg\lvert \int_{t_0}^t \int_M U(\alpha f)\circ \horo_{r} \, \alpha \ell \vol \diff r\Bigg\rvert \ll \|f\|_{W^7} \|\ell\|_{W^7} \, \int_{t_0}^t r^{-2} \diff r  \\
    &\ll \|f\|_{W^7} \|\ell\|_{W^7} \, \frac{t-t_0}{t^2} \ll \nal \, \|f\|_{W^7} \|\ell\|_{W^7} \, t^{-1-\delta_0}.
\end{split}
\]
Since $t_0 / t = 1 + \nal \, t^{-\delta_0}$, from Ratner's quantitative mixing theorem \cite{Rat1}, we conclude
\[
\Bigg\lvert    \int_M f\circ \tc_t \, \ell \volal -  \int_M  (\alpha f) \circ \horo_t \cdot \alpha \ell \, \vol \Bigg\rvert \ll \frac{1}{t} \, \nal^{\frac{1}{2}} \, \|f\|_{W^7} \, \|\ell\|_{W^7},
\]
which proves \Cref{thm:continuity}.

\section{Proofs of \Cref{thm:main_2} and \Cref{prop:correlations_with_small_integral}}

We complete the proof of \Cref{thm:main} by proving \Cref{thm:main_2} and \Cref{prop:correlations_with_small_integral}.

\subsection{Proof of \Cref{thm:main_2}}

We exploit the full expansion of $\int_0^t u \circ \tc_r(x) \diff r$ given by \eqref{eq:integral_horo_tc}, namely
\begin{multline*}
    \Bigg\lvert\int_0^{t} u \circ \tc_r(x) \diff r \\
    - \sum_{\substack{ \mu \in \Spec_{\geq 0}\\ \bullet \in \{+,-\}}}\tau(x,t)^{\lambda_{\bullet}} \, \big[\cP_{\lambda_{\bullet}}(\alpha f) (\geo_{\log \tau(x,t)}(x),t) - \cQ_{\lambda_{\bullet}}(\alpha f) (\geo_{\log \tau(x,t)} \circ \tc_t(x),t)\big] \Bigg\rvert \ll \|u\|_{W^6}. 
\end{multline*}
For $\mu \in \Spec_{\geq 0}$, let us define 
\[
\begin{split}
    &P_{\lambda_{\bullet}}(x,t) = \tau(x,t)^{\lambda_{\bullet}} \cP_{\lambda_{\bullet}}(\alpha u) (\geo_{\log \tau(x,t)}(x),t), \\
    &Q_{\lambda_{\bullet}}(x,t) = \tau(\tc_{-t}x,t)^{\lambda_{\bullet}} \cQ_{\lambda_{\bullet}}(\alpha u) (\geo_{\log \tau(\tc_{-t}x,t)}(x),t) = \tau_{-}(x,t)^{\lambda_{\bullet}} \cQ_{\lambda_{\bullet}}(\alpha u) (\geo_{\log \tau_{-}(x,t)}(x),t),
\end{split}
\]
where the last equality follows from the fact that $\tau(\tc_{-t}x,t) = -\tau(x,-t)$.
Using the invariance of the measure $\volal$, we have
\begin{equation}\label{eq:Azt}
\begin{split}
    &\Bigg\lvert \int_M \Big( \int_0^t u \circ \tc_r \diff r \Big) \cdot f\circ \tc_t \cdot \ell \volal\Bigg\rvert \\
    & \qquad \ll  \sum_{\substack{ \mu \in \Spec_{\geq 0}\\ \bullet \in \{+,-\}}}   
    \Bigg[\Bigg\lvert\int_M f\circ \tc_t \cdot \ell \cdot P_{\lambda_{\bullet}}(x,t) \volal \Bigg\rvert +\Bigg\lvert\int_M  \ell\circ \tc_{-t} \cdot f  \cdot Q_{\lambda_{\bullet}}(x,t) \volal\Bigg\rvert \Bigg] +\|u\|_{W^6} \, \|f\|_{\infty} \, \|\ell\|_{\infty}.
\end{split}
\end{equation}

We will estimate each integral appearing in the right-hand side using \Cref{cor:alternative_FU}. For this reason, we need some estimates on the norms 
$\|P_{\lambda_{\bullet}}(x,t)\|_X$ and $\|Q_{\lambda_{\bullet}}(x,t)\|_X$.

\begin{lemma}\label{lem:first_one}
For all $(\mu, \bullet) \in \Spec_{\geq 0} \times \{+,-\}$, we have
\[
    \|P_{\lambda_{\bullet}}(x,t)\|_X \ll t^{\Re \, \lambda_{\bullet}}\|\cP_{\lambda_{\bullet}} (\alpha u)(x,t)\|_X, \qquad \text{and} \qquad \|Q_{\lambda_{\bullet}}(x,t)\|_X \ll t^{\Re \, \lambda_{\bullet}}\|\cQ_{\lambda_{\bullet}}(\alpha u)(x,t)\|_X.
\]
Furthermore, assume that $\mu_0 < 1/4$, and let $\alpha_0 := 1- \frac{1}{\alpha}$, $\cP_t(x):= \cP_{\lambda_0}(\alpha u)(x,t)$, and $\cQ_t(x):= \cQ_{\lambda_0}(\alpha u)(x,t)$. Then,
\[
\begin{split}
    &|P_{\lambda_0}(x,t) - t^{\lambda_0} \cP_t \circ \geo_{\log t}(x)| \ll  t^{\lambda_0 - \delta_0}\, \|\cP_t\|_{X}, \qquad \text{and} \qquad |Q_{\lambda_0}(x,t) - t^{\lambda_0} \cQ_t \circ \geo_{\log t}(x)| \ll  t^{\lambda_0 - \delta_0}\, \|\cQ_t\|_{X},
\end{split}
\]
and 
\[
\begin{split}
    & \Bigg\lvert XP_{\lambda_0}(x,t) -  t^{\lambda_0}\Bigg[ \alpha_0 \circ \tc_t \cdot \Big( \lambda_0  \cP_t -  X\cP_t \Big) \circ \geo_{\log t} -  X\cP_t\circ \geo_{\log t} \Bigg]
    \Bigg\rvert \ll t^{\lambda_0 - \delta_0} \|\cP\|_{X,2}, \qquad \text{and}\\
    &\Bigg\lvert XQ_{\lambda_0}(x,t) -  t^{\lambda_0}\Bigg[ 
    \alpha_0 \circ \tc_{-t}\cdot \Big( \lambda_0\,  \cQ_t -  X \cQ_t\Big)\circ \geo_{\log t} - X \cQ_t \circ \geo_{\log t} \Bigg] \Bigg\rvert \ll t^{\lambda_0 - \delta_0} \|\cQ\|_{X,2}.
\end{split}
\]
\end{lemma}
\begin{proof}
    Let us consider the case $(\mu,\bullet) \neq (1/4,+)$.
    From the definitions of $P_{\lambda_{\bullet}}(x,t)$ and $Q_{\lambda_{\bullet}}(x,t)$, and from \eqref{eq:t_tau_bound}, it is immediate that 
    \[
        \|P_{\lambda_{\bullet}}(x,t)\|_{\infty} \ll t^{\Re \, \lambda_{\bullet}}\|\cP_{\lambda_{\bullet}}(\alpha u)(x,t)\|_{\infty}, \qquad \text{and} \qquad \|Q_{\lambda_{\bullet}}(x,t)\|_{\infty} \ll t^{\Re \, \lambda_{\bullet}}\|\cQ_{\lambda_{\bullet}}(\alpha u)(x,t)\|_{\infty}.
    \]
    We compute the derivatives
    \[
    \begin{split}
        XP_{\lambda_{\bullet}}(x,t) = &X \Big( \tau(x,t)^{\lambda_{\bullet}} \Big) \cP_{\lambda_{\bullet}}(\alpha u)( \geo_{\log \tau(x,t)}(x),t ) \\
        &+ \tau(x,t)^{\lambda_{\bullet}} X\cP_{\lambda_{\bullet}}(\alpha u) (\geo_{\log \tau(x,t)}(x),t) \Big(1+ X\log \tau(x,t) \Big) \\
        =& \tau(x,t)^{\lambda_{\bullet}} \Bigg[ {\lambda_{\bullet}}  \frac{X\tau(x,t)}{\tau(x,t)} \cP_{\lambda_{\bullet}}(\alpha u) (\geo_{\log \tau(x,t)} (x),t) + X\cP_{\mu}^{\bullet}(\alpha u) (\geo_{\log \tau(x,t)}(x),t) \Big(1+ \frac{X\tau(x,t)}{\tau(x,t)} \Big)\Bigg],
    \end{split}
    \]
    and, similarly,  
    \[
    \begin{split}
        XQ_{\lambda_{\bullet}}(x,t) =& X \Big( {\tau}_{-}(x,t)^{\lambda_{\bullet}} \Big) \cQ_{\lambda_{\bullet}}(\alpha u) (\geo_{\log {\tau}_{-}(x,t)}(x),t) \\
        &+ {\tau}_{-}(x,t)^{\lambda_{\bullet}} X\cQ_{\lambda_{\bullet}}(\alpha u) ( \geo_{\log {\tau}_{-}(x,t)} (x),t) \Big(1+ X\log {\tau}_{-}(x,t) \Big) \\
        =&  {\tau}_{-}(x,t)^{\lambda_{\bullet}} \Bigg[{\lambda_{\bullet}} \frac{X\tau_{-}(x,t)}{\tau_{-}(x,t)} \cQ_{\lambda_{\bullet}}(\alpha u)  (\geo_{\log {\tau}_{-}(x,t)} (x) ,t) + X\cQ_{\lambda_{\bullet}}(\alpha u) (\geo_{\log {\tau}_{-}(x,t)}(x),t) \Big(1+ \frac{X\tau_{-}(x,t)}{\tau_{-}(x,t)} \Big)\Bigg].
    \end{split}
    \]
    \Cref{lem:Xtau_over_tau} and \eqref{eq:t_tau_bound} prove the bounds on $\|P_{\lambda_{\bullet}}(x,t)\|_X$ and $\|Q_{\lambda_{\bullet}}(x,t)\|_X$. 
    The computations for the case $(\mu,\bullet)=(1/4,+)$ are analogous and are left to the reader.

    Let us now assume that $\mu_0 < 1/4$. 
    By \Cref{lem:tau_beta_minus_t_beta},
    \begin{multline*}
        |P_{\lambda_0}(x,t) - t^{\lambda_0} \cP_t \circ \geo_{\log t}(x)|  \\ \ll |\tau(x,t)^{\lambda_0} - t^{\lambda_0} | \cdot \|\cP_t\|_{\infty} +  t^{\lambda_0} \, |\log \tau(x,t) - \log t|\cdot \|X \cP_t\|_{\infty}  \ll  t^{\lambda_0 - \delta_0}\, \|\cP_t\|_{X},
    \end{multline*}
    and, similarly, 
    \begin{multline*}
        |Q_{\lambda_0}(x,t) - t^{\lambda_0} \cQ_t\circ \geo_{\log t}(x)|  \\ \ll |{\tau}_{-}(x,t)^{\lambda_0} - t^{\lambda_0} | \cdot \|\cQ_t\|_{\infty} +  t^{\lambda_0} \, |\log {\tau}_{-}(x,t) - \log t|\cdot \|X \cQ_t\|_{\infty}  \ll  t^{\lambda_0- \delta_0}\, \|\cQ_t\|_{X}.
    \end{multline*}
    Recall that, by \Cref{lem:Xtau_over_tau}, we have 
    \[
        \left\lvert \frac{X\tau_{\pm}(x,t)}{\tau_{\pm}(x, t)} - \alpha_0 \circ \tc_{\pm t}\right\rvert \ll t^{\lambda_0-1}. 
    \]
    Then, from the computations of the derivatives above, we obtain
    \begin{equation*}
        \Bigg\lvert XP_{\lambda_0}(x,t) - t^{\lambda_0}\Big( \lambda_0 \, \alpha_0 \circ \tc_t \, \cP_t \circ \geo_{\log t} +  X\cP_t \circ \geo_{\log t} + \alpha_0 \circ \tc_t \, X\cP \circ \geo_{\log t} \Big)\Bigg\rvert \\
        \ll t^{\lambda_0 - \delta_0} \|\cP\|_{X,2},
    \end{equation*}
    and 
    \begin{equation*}
        \Bigg\lvert XQ_{\lambda_0}(x,t) - t^{\lambda_0}\Big( \lambda_0 \alpha_0 \circ \tc_{-t} \, \cQ_t \circ \geo_{\log t} +  X\cQ_t \circ \geo_{\log t} + \alpha_0 \circ \tc_{-t} \, X\cQ_t \circ \geo_{\log t} \Big)\Bigg\rvert \\
        \ll t^{\lambda_0 - \delta_0} \|\cQ_t\|_{X,2}.
    \end{equation*}       
    This completes the proof.
\end{proof}

\medskip

We go back to the proof of  \Cref{thm:main_2}. We consider two separate cases, depending on the value of $\mu_0$.

\noindent \textbf{Case 1.}
Let us assume that $\mu_0 \geq 1/4$. \Cref{cor:alternative_FU} together with the estimates in \Cref{lem:first_one} yield
\begin{equation*}
\begin{split}
    &   \Bigg\lvert \int_M f\circ \tc_t \, \ell \, P_{\lambda_{\bullet}}(x,t) \volal \Bigg\rvert 
    \ll \frac{R(t)}{t} \|f\|_{W^7} \, \|\ell P_{\lambda_{\bullet}}(x,t)\|_{X}
    \ll \frac{R(t)^2}{t}\|f\|_{W^7} \, \|\ell\|_{X} \|\cP_{\lambda_{\bullet}}(\alpha u)(x,t)\|_{X},\\
    &   \Bigg\lvert \int_M \ell\circ \tc_{-t} \, f \, Q_{\lambda_{\bullet}}(x,t) \volal \Bigg\rvert
    \ll \frac{R(t)}{t} \|\ell\|_{W^7} \, \|fQ_{\lambda_{\bullet}}(x,t) \|_{X}
    \ll \frac{R(t)^2}{t} \|\ell\|_{W^7} \, \|f\|_{X} \|\cQ_{\lambda_{\bullet}}(\alpha u)(x,t)\|_{X}.
\end{split}
\end{equation*}
Substituting into \eqref{eq:Azt}, we conclude 
\begin{equation*}
\begin{split}
    \Bigg\lvert \int_M \Big( \int_0^t u \circ \tc_r \diff r \Big) \cdot f\circ \tc_t \, \ell \volal \Bigg\rvert   & \ll \|f\|_{W^7} \, \|\ell\|_{W^7} \frac{R(t)^2}{t}
    \sum_{\substack{ \mu \in \Spec_{\geq 0} \\ \bullet \in \{+,-\}}} \big[ \|\cP_{\lambda_{\bullet}}(\alpha u)(x,t)\|_X + \|\cQ_{\lambda_{\bullet}}(\alpha u)(x,t)\|_X\big] \\
    &\ll  \|u\|_{W^7} \, \|f\|_{W^7} \, \|\ell\|_{W^7} \frac{R(t)^2}{t}.
\end{split}
\end{equation*}

\medskip

\noindent \textbf{Case 2.}
We conclude by treating the case $\mu_0<1/4$. 
By \Cref{cor:alternative_FU} and  \Cref{lem:first_one}, for all $(\mu, \bullet) \in \Spec_{\geq 0} \times \{+,-\}$, with $ (\mu, \bullet)  \neq (\mu_0,+)$, we have
\begin{equation*}
\begin{split}
    &\Bigg\lvert \int_M f\circ \tc_t \, \ell \, P_{\lambda_{\bullet}}(x,t) \volal \Bigg\rvert \ll 
    t^{-\frac{1-\nu_0}{2}}\|f\|_{ W^6} \, \|\ell P_{\lambda_{\bullet}}(x,t)\|_{X}
    \ll \|f\|_{ W^6} \, \|\ell\|_{X} \|\cP_{\lambda_{\bullet}}(\alpha u)(x,t)\|_X \,t^{\nu_0 -\delta_0}, \\ 
    &\Bigg\lvert \int_M  \ell\circ \tc_{-t} \, f  \, Q_{\lambda_{\bullet}}(x,t) \volal  \Bigg\rvert \ll 
    t^{-\frac{1-\nu_0}{2}}\|\ell\|_{ W^6} \, \|f Q_{\lambda_{\bullet}}(x,t)\|_{X}   \ll \|f\|_{X} \, \|\ell\|_{W^6} \|\cQ_{\lambda_{\bullet}}(\alpha u)(x,t)\|_X \, t^{\nu_0 -\delta_0};
\end{split}
\end{equation*}
Thus, from \eqref{eq:Azt} we deduce
\begin{equation}\label{eq:Azt_second}
\begin{split}
    \Bigg\lvert \int_M \Big( \int_0^t u \circ \tc_r \diff r \Big) \cdot f\circ \tc_t \, \ell \volal \Bigg\rvert    \ll &
    \Bigg\lvert \int_M f\circ \tc_t \, \ell \, P_{\lambda_0}(x,t) \volal  \Bigg\rvert +
    \Bigg\lvert\int_M  \ell\circ \tc_{-t} \, f  \, Q_{\lambda_0}(x,t) \volal  \Bigg\rvert \\
    &+ \|u\|_{W^6} \,
    \|f\|_{W^6} \, \|\ell\|_{W^6}t^{\nu_0 -\delta_0}.
\end{split}
\end{equation}
It remains to study the two integrals in the right-hand side above. In the rest of the proof, we simplify the notation by setting $P_t(x) =P_{\lambda_0}(x,t)$, $\cP_t = \cP_{\lambda_0}(\alpha u)(\cdot, t)$, as well as $Q_t(x) =Q_{\lambda_0}(x,t)$, and $\cQ_t = \cQ_{\lambda_0}(\alpha u)(\cdot,t)$.
    
Let us start from the first one. 
By \Cref{cor:alternative_FU} with $w=f$ and $v=\ell \,  P_t$, we deduce that there exist H\"{o}lder function $\cF_t$ and $\widehat{\cF}_t$ (depending on $f$ only) such that
\[
\begin{split}
    &\Bigg\lvert \int_M f\circ \tc_t \, \ell \, P_t \volal  \Bigg\rvert \ll t^{\lambda_0 - 1} \Bigg[\Bigg\lvert \int_M f\circ \tc_t \, \ell P_t \, \cA_t^{+} \circ \geo_{\log t} \volal  \Bigg\rvert 
    +\Bigg\lvert \int_M \widehat{\cF}_t \circ \geo_{\log t} \, \ell P_t  \volal  \Bigg\rvert \\
    &\qquad +\Bigg\lvert \int_M {\cF}_t \circ \geo_{\log t} \, \Big[\Big( 1-\frac{X\alpha}{\alpha}\Big)\ell -X\ell \Big] P_t \volal  \Bigg\rvert 
    +\Bigg\lvert \int_M {\cF}_t \circ \geo_{\log t} \, \ell \, XP_t \volal  \Bigg\rvert
     \Bigg] + \|f\|_{W^7} \, \|\ell P_t\|_{X} t^{\lambda_0-1- \delta_0}.
\end{split}
\]
\Cref{lem:first_one} implies
\[
\begin{split}
    &\Bigg\lvert \int_M f\circ \tc_t \, \ell \, P_t \volal  \Bigg\rvert \ll  t^{\nu_0} \Bigg[\Bigg\lvert \int_M f\circ \tc_t \, \ell \, [\cP_t \cdot \cA_t^{+}] \circ \geo_{\log t} \volal  \Bigg\rvert \\
    & \qquad + \Bigg\lvert \int_M \alpha \ell [\cP_t \cdot \widehat{\cF}_t] \circ \geo_{\log t} \vol  \Bigg\rvert 
    + \Bigg\lvert \int_M [\alpha \ell - X(\alpha \ell) ] \cdot [\cP_t \cdot {\cF}_t] \circ \geo_{\log t} \vol  \Bigg\rvert \\
    & \qquad + \Bigg\lvert \int_M \alpha_0 \circ\tc_t \cdot \ell [\cF_t \cdot (\lambda_0 {\cP}_t-X{\cP}_t)] \circ \geo_{\log t} \volal  \Bigg\rvert 
    + \Bigg\lvert \int_M \alpha \ell  \cdot [\cF_t \cdot X{\cP}_t] \circ \geo_{\log t} \vol  \Bigg\rvert \\
    & \qquad + \|f\|_{W^7} \, \|\ell\|_{X} \, \|\cP_t\|_{X,2}t^{\nu_0- \delta_0}.
\end{split}
\]
All the functions appearing in the 5 integrals in the right-hand side above are at least H\"{o}lder continuous of exponent $\beta = \frac{1-\nu_0}{2}$. Thus, up to choosing a smaller $\delta_0$, we have that
\[
\begin{split}
    &\Bigg\lvert \int_M f\circ \tc_t \, \ell \, [\cP_t \cdot \cA_t^{+}] \circ \geo_{\log t} \volal  \Bigg\rvert \ll  \|f\|_{W^6} \, \|\ell\|_{W^6} \, \| \cP\|_{\beta} t^{-\delta_0},\\
    &\Bigg\lvert \int_M \alpha_0 \circ\tc_t \cdot \ell [\cF_t \cdot (\lambda_0 {\cP}_t-X{\cP}_t)] \circ \geo_{\log t} \volal  \Bigg\rvert  \ll \|f\|_{W^6} \, \|\ell\|_{W^6} \, (\| \cP\|_{\beta} + \|X \cP\|_{\beta}) \, t^{-\delta_0},
\end{split}
\]
by \Cref{prop:improved_FU_holder}; whence
\[
\begin{split}
    &\Bigg\lvert \int_M \alpha \ell [\cP_t \cdot \widehat{\cF}_t] \circ \geo_{\log t} \vol  \Bigg\rvert   \ll  \|\cP_t \cdot \cF_t\|_{\beta} \, \|\ell\|_{\beta} \, t^{-\delta_0},\\
    &  \Bigg\lvert \int_M [\alpha \ell - X(\alpha \ell) ] \cdot [\cP_t \cdot {\cF}_t] \circ \geo_{\log t} \vol  \Bigg\rvert \ll \|\cP_t \cdot \cF_t\|_{\beta} \, (\|\ell\|_{\beta} + \|X\ell\|_{\beta}) \,  t^{-\delta_0},\\
    &\Bigg\lvert \int_M \alpha \ell  \cdot [\cF_t \cdot X{\cP}_t] \circ \geo_{\log t} \vol  \Bigg\rvert \ll \|X\cP_t \cdot \cF_t\|_{\beta} \, \|\ell\|_{\beta} \,  t^{-\delta_0},
\end{split}
\]
by Ratner's Theorem on exponential mixing of the geodesic flow for H\"{o}lder functions \cite[Theorem 2]{Rat1}.
We then conclude
\[
    \Bigg\lvert \int_M f\circ \tc_t \, \ell \, P_t \volal  \Bigg\rvert \ll \|u\|_{W^7} \,\|f\|_{W^7} \, \|\ell\|_{W^7} t^{\nu_0 - \delta_0}.
\]

The analogous bound on the second integral in the right-hand side of \eqref{eq:Azt_second}  is obtained in the same way, with $Q_t$ instead of $P_t$, with the roles of $f$ and $\ell$ reversed, and with $-t$ instead of $t$. The proof is then complete.

\subsection{Proof of \Cref{prop:correlations_with_small_integral}}

Let us define
\[
    \Delta(x) := \tau(x,t) - t_0,
\]
and recall that, by definition of $t_0$, we have $\theta \ll \Delta(x) \ll \theta$. Also, note that $|\Delta(x) - \theta|\ll \nal \, t^{-1}R(t)$.

By \Cref{thm:Rav}, we have
\[
\begin{split}
    \Bigg\lvert \int_M \int_{t_0}^{\tau(x,t)} &w \circ \horo_r(x) \cdot v(x) \diff r \vol \Bigg\rvert = \Bigg\lvert \int_M \int_{0}^{\tau(x,t)-t_0} w \circ \horo_r(\horo_{t_0}x) \cdot v(x) \diff r \vol \Bigg\rvert \\
    \ll & \sum_{\substack{ \mu \in \Spec_{\geq 0} \\ \bullet \in \{+,-\}}} \Bigg[ \Bigg\lvert \int_M \Delta(x)^{\lambda_{\bullet}}%
    \cP_{\lambda_{\bullet}}w (\geo_{\log \Delta(x)} \circ \horo_{t_0} (x), \Delta(x))\cdot v(x)  \vol \Bigg\rvert \\
    & + \Bigg\lvert 
    \int_M \Delta(x)^{\lambda_{\bullet}}
    \cQ_{\lambda_{\bullet}}w( \geo_{\log \Delta(x)} \circ \tc_t (x),\Delta(x))\cdot v(x)  \vol \Bigg\rvert \Bigg]  + \|w\|_{W^6} \|v\|_{\infty}.
\end{split}
\]
Using first \Cref{thm:Rav}-(b) to replace $\Delta(x)$ with $\theta$ in the second variable of $\cP_{\lambda_{\bullet}}w$ and $\cQ_{\lambda_{\bullet}}w$, and subsequently by measure invariance, we rewrite
\begin{equation}\label{eq:last_eq}
\begin{split}
    \Bigg\lvert \int_M \int_{t_0}^{\tau(x,t)} &w \circ \horo_r(x) \cdot v(x) \diff r \vol \Bigg\rvert 
    \\
    \ll & \sum_{\substack{ \mu \in \Spec_{\geq 0} \\ \bullet \in \{+,-\}}} \Bigg[ \Bigg\lvert \int_M \Delta(\horo_{-t_0}x)^{\lambda_{\bullet}}%
    \cP_{\lambda_{\bullet}}w (\geo_{\log \Delta(\horo_{-t_0}x)} (x), \theta)\cdot v\circ \horo_{-t_0}(x)  \vol \Bigg\rvert \\
    & + \Bigg\lvert 
    \int_M \Delta(\tc_{-t}x)^{\lambda_{\bullet}}
    \cQ_{\lambda_{\bullet}}w( \geo_{\log \Delta(\tc_{-t}x)}  (x), \theta)\cdot \frac{v}{\alpha}\circ \tc_{-t}(x)  \vol \Bigg\rvert \Bigg]  + \|w\|_{W^6} \|v\|_{\infty}.
\end{split}
\end{equation}

By \Cref{rmk:principal_series}, if $w$ is supported only on the principal series, \eqref{eq:last_eq} can be replaced by
\begin{equation}\label{eq:last_eq_principal_series}
\begin{split}
    \Bigg\lvert \int_M &\int_{t_0}^{\tau(x,t)} w \circ \horo_r(x) \cdot v(x) \diff r \vol \Bigg\rvert 
    \\
    \ll & \sum_{\substack{ \mu \in \Spec_{\geq 0} \\ \bullet \in \{+,-\}}} \Bigg[ \Bigg\lvert \int_M \Delta(\horo_{-t_0}x)^{\lambda_{\bullet}}%
    \cP_{\lambda_{\bullet}}w (\geo_{\log \Delta(\horo_{-t_0}x)} (x), \theta)\cdot v\circ \horo_{-t_0}(x)  \vol \Bigg\rvert \\
    & + \Bigg\lvert 
    \int_M \Delta(\tc_{-t}x)^{\lambda_{\bullet}}
    \cQ_{\lambda_{\bullet}}w( \geo_{\log \Delta(\tc_{-t}x)}  (x), \theta)\cdot \frac{v}{\alpha}\circ \tc_{-t}(x)  \vol \Bigg\rvert \Bigg]  + \frac{R(t)}{t} \, \nal \,\|w\|_{W^6} \|v\|_{\infty}.
\end{split}
\end{equation}

\begin{lemma}
We have
\[
\begin{split}
    &\|\Delta(\horo_{-t_0}x)^{\lambda_{\bullet}}%
    \cP_{\lambda_{\bullet}}w ( \geo_{\log \Delta(\horo_{-t_0}x)} (x), \theta)\|_{\infty} \ll \theta^{\Re \lambda_{\bullet}} \| \cP_{\lambda_{\bullet}}w \|_{\infty}, \\
    &\|\Delta(\tc_{-t}x)^{\lambda_{\bullet}}
    \cQ_{\lambda_{\bullet}}w  (\geo_{\log \Delta(\tc_{-t}x)} (x), \theta)\|_{\infty} \ll \theta^{ \Re \lambda_{\bullet}} \| \cQ_{\lambda_{\bullet}}w \|_{\infty},
\end{split}
\]
and 
\[
\begin{split}
    &\|\Delta(\horo_{-t_0}x)^{\lambda_{\bullet}}%
    \cP_{\lambda_{\bullet}}w ( \geo_{\log \Delta(\horo_{-t_0}x)} (x), \theta)\|_{X} \ll \theta^{\Re \lambda_{\bullet} -1} t  \nal \, \| \cP_{\lambda_{\bullet}}w \|_{X}, \\
    &\|\Delta(\tc_{-t}x)^{\lambda_{\bullet}}
    \cQ_{\lambda_{\bullet}}w  (\geo_{\log \Delta(\tc_{-t}x)} (x), \theta)\|_{X} \ll \theta^{\Re \lambda_{\bullet} -1}t \nal \,   \| \cQ_{\lambda_{\bullet}}w  \|_{X}.
\end{split}
\]
\end{lemma}
\begin{proof}
    The claim on the $\| \cdot \|_{\infty}$ norm is immediate since $|\Delta(x)|\ll \theta$. Let us estimate the $X$-derivative.
    By \Cref{lem:Xtau_over_tau}, we have
    \[
        \|X[\cP_{\lambda_{\bullet}}w ( \geo_{\log \Delta(\horo_{-t_0}x)} (x), \theta)]\|_{\infty} \ll \|X\cP_{\lambda_{\bullet}}w \|_{\infty} (1+ \|X \log[\Delta(\horo_{-t_0}x)]\|_{\infty} ) \ll \frac{t}{\theta} \nal \, \|X\cP_{\lambda_{\bullet}}w \|_{\infty}, 
    \]
    and, similarly,
    \[
        \|X[\cQ_{\lambda_{\bullet}}w  (\geo_{\log \Delta(\tc_{-t}x)} (x), \theta)]\|_{\infty} \ll \|X\cQ_{\lambda_{\bullet}}w \|_{\infty} (1+ \|X \log[\Delta(\tc_{-t}x)]\|_{\infty} ) \ll \frac{t}{\theta} \nal \, \|X\cQ_{\lambda_{\bullet}}w \|_{\infty}. 
    \]
    Using the same lemma, we can also estimate
    \[
        \|X[\Delta(\horo_{-t_0}x)^{\lambda_{\bullet}}]\|_{\infty} \ll \|\Delta(x)^{\lambda_{\bullet}}\|_{\infty} \cdot \left\| \frac{X[\tau(\horo_{-t_0}x,t)-t_0]}{\tau(\horo_{-t_0}x,t)-t_0} \right\|_{\infty} \ll \theta^{ \Re \lambda_{\bullet}-1}t \nal,
    \]
    as well as 
    \[
        \|X[\Delta(\tc_{-t}x)^{\lambda_{\bullet}}]\|_{\infty} \ll \|\Delta(x)^{\lambda_{\bullet}}\|_{\infty} \cdot \left\| \frac{X[\tau_{-}(x,t)-t_0]}{\tau_{-}(x,t)-t_0} \right\|_{\infty} \ll \theta^{\Re \lambda_{\bullet}-1}t \nal.
    \]
    Combining the estimates above proves the result.
\end{proof}

We fix $\sigma = \theta/ (\nal \, t)$ and we apply \Cref{prop:improved_FU} to each summand in \eqref{eq:last_eq}; we deduce
\[
\begin{split}
    \Bigg\lvert \int_M \int_{t_0}^{\tau(x,t)}& w \circ \horo_r(x) \cdot v(x) \diff r \vol \Bigg\rvert \ll  \|v\|_{W^6} \|w\|_{W^6} \\
    &+\|v\|_{W^6} \Big( \frac{(\theta\sigma t)^{\lambda_0}}{\sigma t} + \nal \theta^{\lambda_0 - 1} t \frac{(\sigma t)^{\lambda_0}}{t}\Big)\sum_{\substack{ \mu \in \Spec_{\geq 0} \\ \bullet \in \{+,-\}}} \big[ \| \cP_{\lambda_{\bullet}}w  \|_{X} + \| \cQ_{\lambda_{\bullet}}w  \|_{X}\big]
    \\
    \ll & 
    \|w\|_{W^6}\, \|v\|_{W^6}\, \theta^{\nu_0} (1+ \one_{\Spec_{+}}(1/4) \cdot (\log t)^2). 
\end{split}
\]
Again, if $\mu_0 > 1/4$ and if $w$ is supported only on the principal series, using \eqref{eq:last_eq_principal_series} instead of \eqref{eq:last_eq}, the same reasoning as above yields
\[
    \Bigg\lvert \int_M \int_{t_0}^{\tau(x,t)} w \circ \horo_r(x) \cdot v(x) \diff r \vol \Bigg\rvert \ll  \nal^{\frac{1}{2}}\, \|w\|_{W^6}\, \|v\|_{W^6}. 
\]
The proof is then complete.

\section{Correlations of coboundaries}

In this last section, by applying a softer method, we derive some (non-optimal) estimates on the correlations of coboundaries, proving \Cref{thm:main_spectral}. 
Let us fix $f,\ell \in W^6(M)$, with $\volal(f) = \volal(\ell) =0$.
As in the proof of \Cref{lem:key_lemma}, we have
\begin{equation}\label{eq:corr_cob}
\begin{split}
    \int_M \Ual f \circ \tc_t \cdot \ell \volal =& \int_M A_t \cdot \Ual f \circ \tc_t \cdot \frac{\alpha \ell}{A_t} \vol = \int_M \big[ X( f \circ \tc_t) - X f \circ \tc_t\big] \cdot \frac{\alpha \ell}{A_t} \vol \\
    =& - \int_M f \circ \tc_t \cdot X\Big(\frac{\alpha \ell}{A_t} \Big) \vol - \int_M Xf \circ \tc_t \cdot \frac{\alpha \ell}{A_t} \vol \\
    =&\int_M f \circ \tc_t \cdot \alpha \ell \cdot \frac{XA_t}{A_t^2} \vol - \int_M f \circ \tc_t \cdot \frac{X(\alpha \ell)}{A_t} \vol- \int_M Xf \circ \tc_t \cdot \frac{\alpha \ell}{A_t} \vol.
\end{split}
\end{equation}
We estimate the three integrals in the right-hand side separately. 
By \Cref{thm:FU_mixing} and \Cref{lem:estimate_At}, we can bound
\[
    \Bigg\lvert \int_M f \circ \tc_t \cdot \frac{X(\alpha \ell)}{A_t} \vol\Bigg\rvert \ll \|f\|_{W^6} \left\| \frac{X(\alpha \ell)}{\alpha A_t} \right\|_X \frac{R(t)}{t} \ll \|f\|_{W^6} \, \|\ell\|_{X,2} \frac{R(t)}{t^2}.
\]
Similarly, using measure-invariance and the fact that $A_t\circ \tc_{-t}(x) = -A_{-t}(x)$,
\[
\begin{split}
    \Bigg\lvert \int_M Xf \circ \tc_t \cdot \frac{\alpha \ell}{A_t} \vol\Bigg\rvert  &= \Bigg\lvert \int_M Xf \cdot \frac{\ell}{A_t}\circ \tc_{-t} \volal \Bigg\rvert  = \Bigg\lvert \int_M \ell \circ \tc_{-t} \cdot \frac{Xf}{-A_{-t}} \volal \Bigg\rvert
    \ll \|\ell\|_{W^6} \left\| \frac{Xf}{A_{-t}}\right\|_X \frac{R(t)}{t} \\
    &\ll \|\ell\|_{W^6} \, \|f\|_{X,2} \frac{R(t)}{t^2}.
\end{split}
\]
It remains thus to bound the first integral in the right-hand side of \eqref{eq:corr_cob}.
Let us call $\widetilde{\alpha} = 1 - 2 \frac{X\alpha}{\alpha} + \frac{X^2\alpha}{\alpha}$. Then, by \eqref{eq:X_deriv_u}, we have
\[
XA_t = A_t \cdot \Big( 1 - \frac{X\alpha}{\alpha} \Big) \circ \tc_t - \int_0^t \widetilde{\alpha}\circ \tc_r \diff r.
\]
We can therefore rewrite
\[
\begin{split}
    \int_M f \circ \tc_t \cdot \alpha \ell \cdot \frac{XA_t}{A_t^2} \vol &= \int_M \Big[\Big(1 - \frac{X\alpha}{\alpha}\Big)\cdot f\Big] \circ \tc_t \cdot \frac{\ell}{A_t} \volal - \int_M f \circ \tc_t \cdot \frac{\ell}{A_t^2} \Big( \int_0^t \widetilde{\alpha}\circ \tc_r \diff r\Big) \volal \\
    &= \int_M \ell \circ \tc_{-t} \cdot \frac{(1 - X\alpha/ \alpha ) f}{-A_{-t}} \volal - \int_M f \circ \tc_t \cdot \frac{\ell}{A_t^2} \Big( \int_0^t \widetilde{\alpha}\circ \tc_r \diff r\Big) \volal. 
\end{split}
\]
By \Cref{lem:estimate_At}, we have
\[
    \left\|\frac{1}{A_t^2} \Big( \int_0^t \widetilde{\alpha}\circ \tc_r \diff r\Big)  \right\|_X \ll \frac{1}{t},
\]
thus, using \Cref{thm:FU_mixing} exactly as above, we deduce
\[
    \Bigg\lvert \int_M f \circ \tc_t \cdot \alpha \ell \cdot \frac{XA_t}{A_t^2} \vol \Bigg\rvert\ll \|f\|_{W^6} \, \|\ell\|_{W^6} \frac{R(t)}{t^2}.
\]
Substituting back into \eqref{eq:corr_cob}, we conclude
\begin{equation}\label{eq:final_one}
    \Bigg\lvert\int_M \Ual f \circ \tc_t \cdot \ell \volal\Bigg\rvert\ll \|f\|_{W^6} \, \|\ell\|_{W^6} \frac{R(t)}{t^2}.
\end{equation}
We complete the proof of \Cref{thm:main_spectral} by showing that the estimate above implies the result on the spectral measures. Let us fix $f \in W^6(M)$. For any $\xi \neq 0$ and $\varepsilon \in (0, |\xi|/2)$, by \cite[Equations (20)--(23)]{FU}, we have
\[
\begin{split}
    |\sigma_f(\xi - \varepsilon, \xi + \varepsilon)| &\ll \frac{\varepsilon^{2}}{\xi^2} \left\| \int_0^{\varepsilon^{-1}} e^{i\xi t} \Ual f \circ \tc_t \diff t \right\|_{2}^2 \ll \frac{\varepsilon^{2}}{\xi^2} \int_0^{\varepsilon^{-1}}\int_0^{\varepsilon^{-1}} \Bigg\lvert \int_M \Ual f \circ \tc_r \cdot \Ual f \circ \tc_t \volal\Bigg\rvert \diff r \diff t \\
    &\ll \frac{\varepsilon^{2}}{\xi^2} \int_0^{\varepsilon^{-1}} (\varepsilon^{-1} - t) \Bigg\lvert \int_M \Ual f \circ \tc_t \cdot \Ual f \volal\Bigg\rvert \diff t \\
    &\ll \|\Ual f\|_{\infty}^2 \frac{\varepsilon}{\xi^2}  + \frac{\varepsilon^{2}}{\xi^2} \int_2^{\varepsilon^{-1}} (\varepsilon^{-1} - t) \Bigg\lvert \int_M \Ual f \circ \tc_t \cdot \Ual f \volal\Bigg\rvert \diff t.
\end{split}
\]
\Cref{eq:final_one} yields 
\[
\int_2^{\varepsilon^{-1}} (\varepsilon^{-1} - t) \Bigg\lvert \int_M \Ual f \circ \tc_t \cdot \Ual f \volal\Bigg\rvert \diff t \ll \| f\|_{W^6}^2 \int_2^{\varepsilon^{-1}} (\varepsilon^{-1} - t)  \frac{R(t)}{t^2} \diff t \ll \| f\|_{W^6}^2\varepsilon^{-1},
\]
hence the proof is complete.

\end{document}